\documentclass[a4paper,10pt]{article}
\usepackage{amsmath,amssymb,amsthm,latexsym,graphicx}

\theoremstyle{plain}
\newtheorem{lem}{Lemma}
\newtheorem{prop}[lem]{Proposition}
\newtheorem{thm}[lem]{Theorem}

\newtheorem{rul}{Rule}
\theoremstyle{definition}
\newtheorem*{defn}{Definition}

\theoremstyle{remark}
\newtheorem*{rem}{Remark}

\def\Z{\mathbb Z}
\def\Q{\mathbb Q}

\newcommand{\ord}{{\mathrm{ord}}}

\newcommand{\sing}{\mathop{\mathrm{Sing}}}
\def\OO{{\cal O}}
\def\OC{\widehat{\cal O}}
\def\GG{{G}}
\def\cc{z}

\def\factor{monomial factor}
\def\factors{monomial factors}
\def\Factors{Monomial factors}
\def\factorization{quotient}
\title{A Game for the Resolution of Singularities}

\author{Herwig Hauser and Josef Schicho\thanks{ This research was supported by the Austrian Science Fund (FWF) in the frame of the  research project P-21461.} }

\begin{document}

\maketitle

%-------------------------------------------------

\begin{abstract}
\noindent 

We propose a combinatorial game on finite graphs, called {\it Salmagundy}, that is played by two protagonists, Dido and Mephisto. The game captures the logical structure of a proof of the resolution of singularities. In each round, the graph of the game is modified by the moves of the players. When it assumes a final configuration, Dido has won. Otherwise, the game goes on forever, and nobody wins. In particular, Mephisto cannot win himself, he can only prevent Dido from winning.

We show that Dido always possesses a winning strategy, regardless of the initial shape of the graph and of the moves of Mephisto. This implies -- translating back to algebraic geometry -- that there is a choice of centers for the blowup of singular varieties  in characteristic zero which eventually leads to their resolution. The algebra needed for this implication is elementary. The transcription from varieties to graphs and from blowups to modifications of the graph thus axiomatizes the proof of the resolution of singularities. In principle, the same logic could also work in positive characteristic, once an appropriate descent in dimension is settled.
\end{abstract}

%-------------------------------------------------
%            INTRODUCTION
%-------------------------------------------------

\addtocounter{section}{1}

\section*{Introduction}

{\small\it
Are you ready? the {\em Umpire} asks -- and {\em Dido} and {\em Mephisto} jibe.
Dido says: I would like to {\em open a quest}. The Umpire unfolds a graph
$\Gamma$ on the desk, with nodes and edges. This will be the {\em main}
quest of your game, you may start to play. Dido nominates a node, the center, and Mephisto responds to it with a modification of the graph, the blowup, replacing the node by a new graph.
}\\

The existence of resolutions of singularities in characteristic zero
has been proven by Hironaka in his landmark paper \cite{Hironaka:64}.
Since then, several authors gave variations and simplifications of
Hironaka's proof. We mention \cite{Villamayor:89, Villamayor:92, Bierstone_Milman:97},
who decribed algorithms which are natural in the sense that they lead to
a resolution which is invariant under any group of automorphisms;
\cite{Encinas_Hauser:02}, which conceptualized the proof by introducing mobiles;
\cite{Wlodarczyk:05}, which simplified the proof by introducing 
homogenized ideals; and \cite{Fruehbis-Krueger_Pfister:03, Bodnar_Schicho:99f},
which describe implementations of the resolution algorithm.

All these proofs have to deal with the complication
that arises from the fact that the construction of coefficient ideals or
hypersurfaces of maximal contact is only locally possible. This then requires to prove that, despite of the local choices, the final resolution is independent of
these choices. 

In the present paper, we combine many of the ideas of earlier papers in order to introduce a new type of resolution datum which seems to be well suited for the inductive
definition of the resolution invariant. At any step of the
algorithm, our resolution problems are globally defined; each
construction is deterministic and does not depend on any local choices.
Consequently, the complication of proving that all centers of blowup
are globally defined does not arise.

Instead of defining an invariant, we interpret resolution as a game between two players. The first (our ``local hero'') attempts to improve the singularities. The second is
some malevolent adversary (the ``demon'') who tries to keep the singularities alive as long as possible. The first player chooses the centers of the blowups, the second provides new order functions after each blowup. The order function defines a stratification of the ambient space. The stratification can be described as a labelled
graph. Choosing the center corresponds to choosing a node of this graph. Each blowup operation produces a new labelled graph satisfying various restrictions. In this way, we may define the moves of the game as modifications of the labelled graph according to those
restrictions.

After the rules of the game are fixed, the proof of Hironaka's result
can be divided into two parts which are logically independent from
each other. First, one has to prove that the game is general enough
to model resolution problems. This part requires algebra,
but the proof is not very complex. Second, one has to prove that there
exists a winning strategy. This part requires to solve a complicated
combinatorial problem, but it does not need any algebra; in principle,
it is an argument about labelled graphs.

Apart from the logical advantage, this approach makes
it also easier to think of other resolution strategies. Hopefully, the algebraic properties of ideals and blowups which make Hironaka's proof work become more lucid, especially for non-experts. In addition, the logical argument could in principle also work in the still unsolved case of positive characteristic.

As a matter of historical correctness, it should be mentioned that
the idea of formulating games related to the resolution problem 
is due to Hironaka himself \cite{Hironaka:67}: he introduced
two such games, the simple and the hard polyhedral game.
The first models resolution for hypersurfaces with generic coefficients,
the second local uniformization of arbitrary hypersurfaces.
M. Spivakovsky gave a winning strategy for the first game 
\cite{Spivakovsky:83} and a counter-example 
for the second, indicating a game that cannot be won \cite{Spivakovsky:82}. This has been extended by S. Bloch and M. Levine in order to bring morphisms of schemes of finite type in good position\cite{Bloch:94, Levine:01}. A recent account on games and resolution has been given in \cite{Zeillinger:06}. 

\ 

The paper consists of four parts:  The first section introduces the game {\em Salmagundy}. Its main ingredient are {\em Scenarios}, which are certain labelled graphs. No algebra appears. The second section introduces the concept of {\em Gallimaufry}, which is our version of a resolution datum. This part uses only basic algebra. The third section establishes the transcription between Gallimaufries and Scenarios: it is shown how to view the resolution process of a Gallimaufry inside the game as an evolution of a Scenario, and, conversely, how any winning strategy for the game implies the existence of a resolution algorithm for singular varieties in characteristic zero. The last section finally shows that the game has a winning strategy, thus completing the proof of resolution.\goodbreak

%-------------------------------------------------
%            COMPARISON OF PROOFS
%-------------------------------------------------

\section*{Comparison of existing proofs}

In the literature there appear at least nine proofs for the resolution of singularities of varieties of arbitrary dimension defined over fields of characteristic zero. In this section we will sketch some of the differences between these proofs.

The original proof of Hironaka is some 200 pages long 
\cite{Hironaka:64}. 
It introduces the key ideas and techniques for all subsequent proofs. As such it has always worked as the principal source of inspiration for mathematicians working in the field. The article is the first paper treating systematically varieties of arbitrary codimension (not just hypersurfaces), it establishes the principle of descent in dimension via local hypersurfaces of maximal contact and coefficient ideals, considers embedded resolution (i.e., aims at normal crossings for the total transform), and develops a multiple, interwoven induction argument between various resolution statements, loc.÷ cit. Note that Hironaka's proof is existential.

In a collaboration with J.M.÷ Aroca and J.L.÷ Vicente, Hironaka then adapted the arguments to the resolution of complex analytic varieties, introducing on the way further techniques 
\cite{Aroca_Hironaka_Vicente:75, Aroca_Hironaka_Vicente:77}. 
The use of the Hilbert-Samuel function as an invariant was made available by B.÷ Bennett 
\cite{Bennett:70}
and then used via normal flatness in several later papers of Hironaka and other authors. In 1989 and 1992, O.÷ Villamayor published two papers which provided a constructive proof of resolution (i.e., indicating the centers of blowups) and added equivariance as a natural further requirement 
\cite{Villamayor:89, Villamayor:92}. 
These two papers were still very complicated and are hard to read, but the main technical advances already appear there (see the appendix in 
\cite{Encinas_Hauser:02} 
for precise references). About the same time, E.÷ Bierstone and P.÷ Milman started to present their approach to resolution. Whereas the first papers 
\cite{Bierstone_Milman:90,Bierstone_Milman:91} 
developed the general ideas and concepts, the article 
\cite{Bierstone_Milman:97} 
offers a complete and thorough presentation with all technical details. It is still complicated to read, in part due to the consideration of certain equivalence relations for ideals in order to construct  global objects from local data. For an extensive comparison of the papers of Villamayor and Bierstone-Milman, see their featured Math.÷ Reviews by J.÷ Lipman, respectively H.÷ Hauser. 

The Working Week on Resolution of Singularities in Obergurgl 1997 enforced the renewed interest and activity in the field. Soon after, G.÷ Bodn\'ar and J.÷ Schicho came up with an implementation of Villamayor's algorithm in Maple 
\cite{Bodnar_Schicho:99f}, 
S.÷ Encinas and O.÷ Villamayor succeeded to clarify further the algorithm 
\cite{Encinas_Villamayor:97b}, 
and S.÷ Encinas and H.÷ Hauser gave a very succint proof (just 20 pages) relying on the language of mobiles 
\cite{Encinas_Hauser:02}. 
There, for the first time, a global resolution datum was constructed without using equivalence relations. This was appropriate to define in an intrinsic way the local resolution invariant for the induction. Even though the definition of a mobile and of the invariant is somewhat involved, the advantage is convincing since the resulting proofs become very short and almost automatic. Moreover, the paper clearly distinguishes the places where the characteristic zero assumption enters the scene. 

Next, D.÷ Cutkosky published a book on resolution of singularities, taking up Villamayor's approach 
\cite{Cutkosky:04}. 
At about the same time, J.÷ W\l odarczyk proposed a variation of the descent in dimension, using homogenized coefficient ideals
\cite{Wlodarczyk:05}. 
Up to analytic isomorphism, they are independent of the local choice of a hypersurface of maximal contact. This allowed to show by different methods than in 
\cite{Villamayor:92, Bierstone_Milman:97, Encinas_Hauser:02} 
that the order of the coefficient ideal was well defined. J.÷ Koll\'ar profited of this construction to eliminate the use of a resolution invariant from the proof he presented in his book 
\cite{Kollar:07} 
(nevertheless, it is used implicitly). In 2003, A.÷ Fr\"uhbis-Kr\"uger and G.÷ Pfister published a refined version of the implementation of Bodn\'ar and Schicho in Singular 
\cite{Fruehbis-Krueger_Pfister:03}.

Recently, O.÷ Villamayor and his collaborators A.÷ Bravo, A.÷ Benito and S.÷ Encinas have developed a new descent in dimension, following Jung's method 
\cite{Bravo_Villamayor:10, Benito_Villamayor:10, Encinas_Villamayor:06}. 
They replace restrictions by projections. The use of elimination algebras and differential operators provides a particularly elegant argument in characteristic zero, and opens some options for positive characteristic, as is shown in the resulting proof of resolution for surfaces \cite{Benito_Villamayor:10}.  On the other side, E.÷ Bierstone, P.÷ Milman, and M.÷ Temkin extended resolution to a quite general setting with a strong focus on functorial properties \cite{Bierstone_Milman_Temkin:09}.

All this activity has been complemented in the last years by various proposals and attacks in characteristic $p>0$ which will not be commented here. We refer the interested reader to \cite{Hauser:10}.\\

Let us briefly describe how the present paper embeds into this landscape: From the reasoning in 
\cite{Encinas_Hauser:02} 
it became clear that the logical part of the argument lives somewhat separated from the algebraic part. Said differently, the actual algebraic construction of the descent in dimension (via hypersurfaces of maximal contact and coefficient ideals) as well as the construction of the transversality ideal (to ensure that the chosen center is always transversal to the exceptional divisor) and the companion ideal (to ensure that the singular locus of the coefficient ideal is contained in the singular loucs of the original ideal) did not matter so much as long as these objects satisfied some specific relations between them. It was then a natural step to isolate these properties and to formulate our game in a purely combinatorial manner (even though working out the technical details is kind of intricate). The game shows perfectly the logical structure of Hironaka's proof, which is -- at least to us -- of dazzling beauty. 

Of course, the game becomes only valid if it can be shown that the actual resolution process for a singularity is mimiqued by a winning strategy for it. This goes in two directions: First, one has to translate the algebraic situation to the context of the game, and second, the winning strategy of the game has to be translated back to prove that the resolution process terminates. This is done in the section {\em Transcription}. In principle, also other algebraic constructions or formulations of resolution problems may fit into the game, even in positive characteristic.  

Let us make clear that we do not overestimate the impact or importance of the present paper. After all, it is just another reading of the existing proofs for resolution (of which we have taken up freely many ideas and concepts). But as the combinatorial and algebraic part can be accessed easily even by non-experts, the paper may help to understand better the existing proofs of resolution of singularities.

\

{\it Acknowledgements.} The present paper owes a lot to prior work on the subject, especially of Zariski, Abhyankar, Hironaka, Villamayor, Bierstone-Milman, Cossart, Cutkosky, Encinas, W\l odarczyk, Koll\'ar, Kawanoue-Matsuki, Bodn\'ar, Panazzolo and Fr\"uhbis-Kr\"uger. We are indebted to all these sources and apologize in advance that precise credits could only be provided in a limited number of cases.

%-------------------------------------------------
%              SALMAGUNDY
%-------------------------------------------------

\setcounter{section}{0}

\section{Salmagundy} \label{sec:rules}

In this section we will introduce a combinatorial game, called {\em Salmagundy}.\footnote{{\it Salmagundy}: A 17th \& 18th century composed salad of cold chicken with anchovies, boiled eggs, green beans, boiled onions, grapes, and dressed with a vinaigrette (American Heritage Dictionary). From fr.÷ {\em Salmigondis}: Assemblage disparate, m\'elange confus de choses ou de personnes. Ramassis d'id\'ees, de paroles ou d'\'ecrits formant un tout disparate et incoh\'erent (www.absurditis.com/salmigondis/708).} 
It exhibits the axiomatic and logical structure of the existing proofs for the resolution of singularities of algebraic varieties in characteristic zero. The resolution is typically built on a sequence of blowups in smooth centers which are chosen as the smallest stratum of a suitable stratification of the variety. The choice of the stratification and the proof of termination of the resolution procedure are both established by induction on the ambient dimension. The main focus lies here on the scrutiny of the ideal defining the variety in the ambient space, together with its transforms under blowup.

In the terminology of the game, there will be no references to algebraic concepts such as ideals, varieties or blowup maps. Only labelled graphs -- which evolve along the game -- appear. We will indicate in paralipomena the respective algebraic analogues of the various tokens of the game.

In the section {\em Transcription} we show how to pass from the algebraic setting of a resolution problem for singular varieties -- encoded in a resolution datum called {\em Gallimaufry} -- to the game Salmagundy and how, going back, a winning strategy for the game ensures  the termination of the resolution algorithm of the variety.

Let us first sketch the overall idea of the game. The precise description will start with the subsection {\em Boards}.\footnote{As a general guideline, we try to keep the exposition slim so as to transmit the essential flavor of the various constructions instead of hiding them behind tedious technicalities. It is therefore preferable to accept a rough understanding on a first reading.}

Our game carries on a collection of \emph {scenarios} $C$ -- finite, 
directed graphs whose nodes come with certain labels -- which evolve 
with the moves of the two players, Dido and Mephisto. A round of the game consists 
in a move of Dido, followed by a move of Mephisto. The combination of both moves modifies the actual scenarios, introduces new or deletes existing ones. The rules of the game and the moves of the players in one round 
determine the collection of scenarios for the next round. 

The game starts with a single scenario, the \emph{initial} scenario $C$. 
It is provided by the Umpire. After inspection of the scenario, Dido 
announces her move, to which Mephisto responds with his move. This 
constitutes the first round. There are two types of moves for Dido, {\it blowups} and {\em quests}.

%-------------------------------------------------

\paragraph{Blowups.} A blowup move of Dido is given by the selection of a certain node of the graph, the {\em center}. The {\em response move} of Mephisto consists in modifying all existing scenarios.  This response depends on the chosen center, but allows Mephisto some flexibility on how to change the scenarios.  Each round with blowup moves produces a modification $C'$ of all scenarios $C$, called the {\em blowup transform} of $C$. These transforms are then the {\em actual} scenarios of the game in round two. If the next move of Dido is again a blowup, they will be modi\-fied once more. The evolution of a scenario $C$ continues like this either forever or until it reaches a \emph{final} configuration, which is characterized by certain properties. In this case, Dido has won. If the game does not come to an end, neither player has won (so that Mephisto can never win  -- he can only prevent Dido from winning).

Final scenarios can be characterized by the absence of singularities.
In order to survive, Mephisto needs to claim that there are still
singularities, no matter what the other player is doing. According to \cite{Goethe}, ``Mephisto ist der Geist, der stets verneint'':

\begin{quote}
Der Herr: Hast du mir weiter nichts zu sagen? \\
Kommst du nur immer anzuklagen? \\
Ist auf der Erde ewig dir nichts recht? \\
Mephisto:
Nein, Herr! Ich find' es dort, wie immer, herzlich schlecht.
\end{quote}

%-------------------------------------------------

\paragraph{Quests.} 

The second type of Dido's moves -- {\em calling} or {\em opening a quest} -- results in the creation of new scenarios.  A call can only be placed by Dido, and Mephisto's response to it is a scenario of a new, subordinate {\em quest}. A quest can be intuitively thought of as a subgame of the main game.\footnote{For instance, taking a pawn may be considered as a subgame in chess, similarly as winning a set in lawn tennis.} In our context, a quest is abstractly defined as the -- possibly infinite -- tree of all scenarios that can be obtained from one given scenario by blowup moves of the players.\footnote{See the respective subsection {\em Transforms} below.} Two nodes of the tree are connected 
by a directed edge if the scenario corresponding to the second node is an 
allowed blowup transform of the scenario of the first (see fig.÷ 1).

%-------------------------------------------------
%      PICTURE QUEST
%-------------------------------------------------

\begin{figure}
\begin{center}
\includegraphics[width= 0.85\textwidth ]{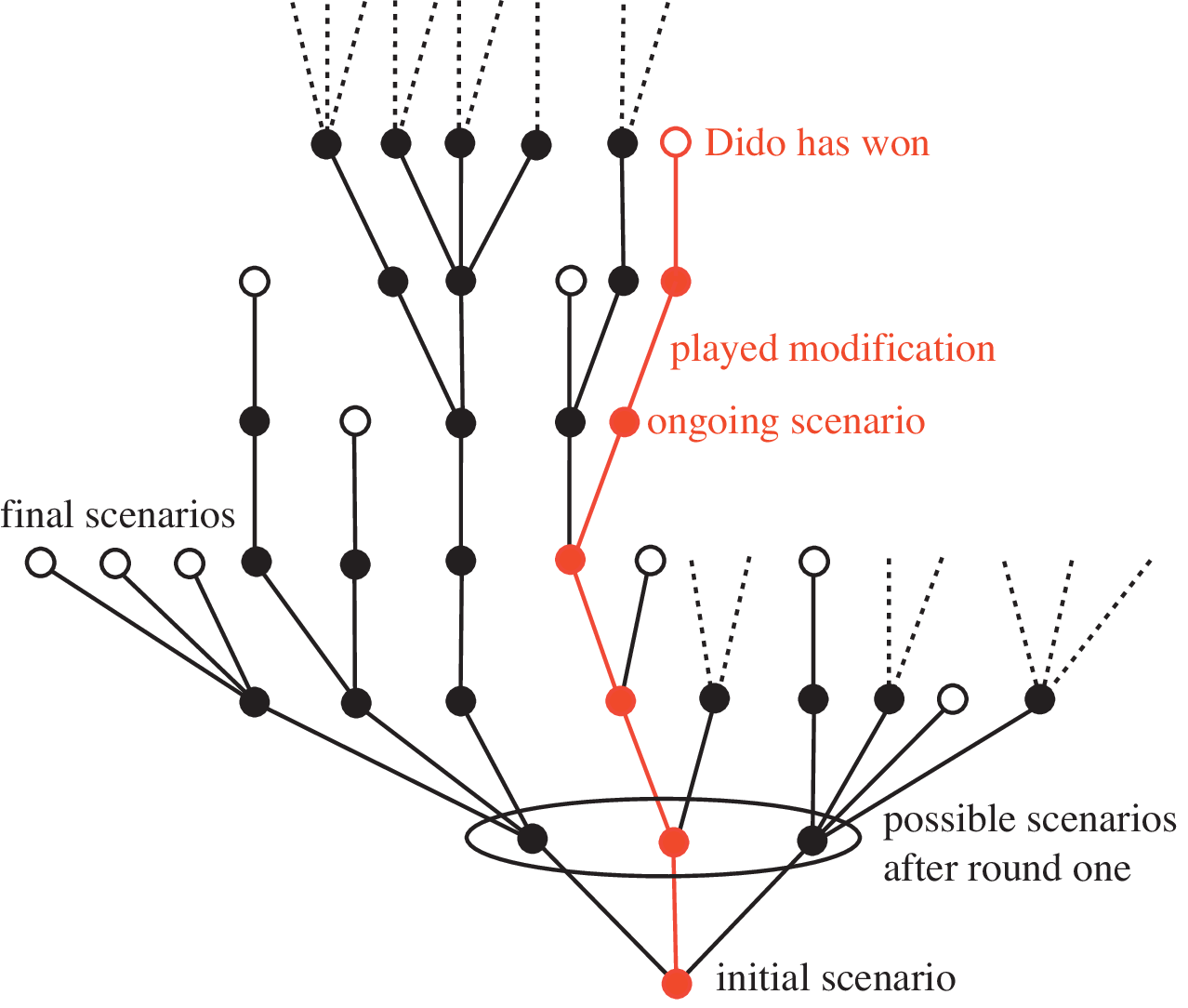}
\caption{ \label{fig:quest} A quest is a tree of scenarios related by blowup moves.}
\end{center}\vskip-.4cm
\end{figure}

%-------------------------------------------------

Mephisto's move, which is considered as a response to Dido's call, consists in 
providing a new scenario for the quest just opened by Dido, as well as to all quests opened at earlier stages of the game. The response scenarios must satisfy certain rules according to the type of the call. In addition, they are subject to certain commutativity rules with respect to the scenarios which existed before the last blowup move of the foregoing rounds.  

The responses of Mephisto are in general not unique, with the exception of a few auxiliary quests, called {\em one way quests}, for which Mephisto has no choice and for which the response scenario is already completely determined by the type of the call and the scenario of the superordinate quest.\footnote{The responses to one way quests can equally be provided by the Umpire.} 

If Dido chooses the first type of move, \emph{blowups}, the actual 
scenarios -- each one being the previous response of Mephisto to a quest -- are split 
into two sets: the first set consists of the scenarios for which the blowup 
is \emph{admissible} in a specified sense. They undergo transformations chosen by Mephisto within certain rules. The remaining scenarios do not transform; their quests are deleted from the game.

It may happen that a subordinate quest reaches earlier a final configuration 
than a superordinate one, in which case the resolved quest is discarded 
from the game. We say that Dido has won the subgame. The rules of the game 
ensure that  the superordinate scenario has then come closer 
(in a precise way) to its final configuration, so winning a subgame 
helps Dido to win the overall game.

The collection of all quests which are at stake before/after a move is called
the set of open quests. Mephisto has to provide responses, say scenarios,
for all open quests. These scenarios are related to each other by precise rules. They share a common underlying structure, the {\em board}, which is a labelled graph. In contrast to other board games, the board in the resolution game will change in each
move of Mephisto.

From a different perspective, opening a quest can also be seen as an 
operational instruction: it is a move of Dido which forces Mephisto 
to respond by scenarios of a given type as long as the quest is open. 
These responses may provide useful information for Dido on how to select 
her next move. 

At each stage of the game, the players have some flexibility of how to choose 
their moves (there are different rules for each of them). The combination of 
Dido's and Mephisto's move defines in a unique way the modification of the 
actual scenarios, the creation of the new scenarios and the deletion of 
certain scenarios (the game is deterministic, chance is excluded). 

Note the difference between the overall game (i.e., the prescription of the 
initial scenario and the collection of rules for transformations and calls) 
and the actually played game (i.e., the sequence of moves applied to the 
initial scenario). We start with a couple of basic ingredients of the game.

%-------------------------------------------------

\paragraph{Boards.}

Mephisto's move in each round consists in responding to all open quests by scenarios. 
These responses share the same underlying structure, the {\em board}.
This is a directed annotated graph  $\Gamma$, which is the Hasse 
diagram of a finite partially ordered set. We say that a node $s$ is {\em below} a node $t$ if $s\leq t$ with respect to the order relation on $\Gamma$; and that $s$ and $t$ are {\em remote} from each other if there is no common node $u$ below both of them.\footnote{Seen geometrically, the nodes $s$ of $\Gamma$ correspond to the strata of a stratification of a manifold $W$ by locally closed subsets. The order relation $s\leq t$ between nodes corresponds to the adjacency of strata, defined by $s$ being in the closure of $t$. Two nodes have a common node below them if and only if the closures of the respective strata intersect.} 
The annotation consists of a non-negative integer $\dim(s)$, the {\em dimension}, for each node $s\in\Gamma$.  It defines a strictly monotonously increasing function on $\Gamma$. The nodes of $\Gamma$ are supposed to have a largest element; its dimension is denoted by $n=\dim(\Gamma)$.\footnote{The number $n$ corresponds to the dimension of the smooth ambient variety of the singularity to be resolved.}

%-------------------------------------------------

\paragraph{Scenarios.} 
A {\em scenario} $C$ on a board $\Gamma$ of dimension $n$ consists of 

\begin{enumerate}
\item 
        two integers $0\le d\le n$ and $0<B$,  the {\em dimension} and the {\em bound};
\item
        subsets ${\cal H}$, ${\cal S}$ and ${\cal T}$ of the set of nodes of $\Gamma$, 
        called the {\em handicap}, the {\em singular set}, and the {\em transversal set};
\item
        a function $\ord:{\cal S}\to\frac{1}{B}{\Bbb Z}\cup\{\infty\}$, the {\em order 
        function};
\item
        a non-empty set $\cal M$ of functions $m:{\cal H}\to\Q_{\ge 0}$, the 
        {\em \factors}. They are extended to $\Gamma$ by 
        $m(s)=\sum_{h \ge s}m(h)$. 
\end{enumerate}

The elements in ${\cal H}$ are called {\em jibs}, those in ${\cal S}$ and ${\cal T}$ {\em singular}, respectively {\em transversal nodes}.\footnote{In the transcription of algebraic resolution problems to the game, the singular set and the order function will come from a gallimaufry. The jibs correspond to the exceptional hypersurfaces, and ${\cal T}$ describes choices of centers which are transversal to these hypersurfaces. The \factors{} stem from exceptional factors which were created by earlier blowups.} See Figure 2 for an example of a scenario with just one monomial factor.\footnote{The dimensions are not indicated in the figure; nodes on the same height will have the same dimension.} 

\goodbreak
%-------------------------------------------------
%          PICTURE SCENARIO 
%-------------------------------------------------

\begin{figure}
\begin{center}
\includegraphics[width= 0.85\textwidth ]{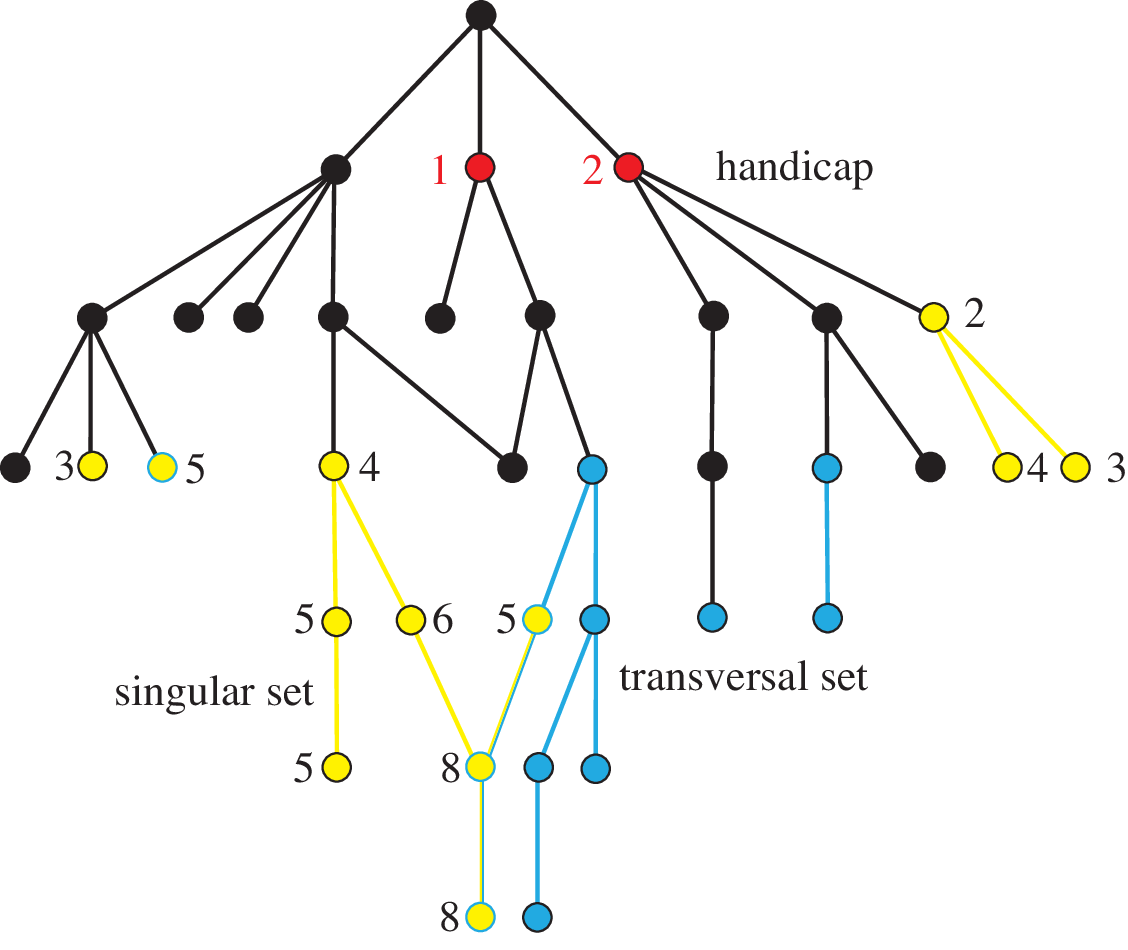}
\caption{ \label{fig:quest} Scenario with order function on $\cal S$ and single \factor{} $m$.}
\end{center}\vskip-.4cm
\end{figure}

%-------------------------------------------------
\goodbreak

\begin{rul} \label{rul:scen} A scenario is subject to the following rules.
\begin{enumerate}
\item Every jib $h\in{\cal H}$ has dimension $n-1$.
\item The set ${\cal S}$ is downward closed.
%\item The order function is weakly monotonously decreasing.
\item \label{srul:max}
        The maximal nodes with order $\infty$ have dimension $d$.
	If ${\cal H}$ is empty, then all these maximal nodes 
	are in ${\cal T}$.
\item \label{srul:out}
	For all $s\in{\cal S}$, we have $\dim(s)\le d$.
	If $\dim(s)=d$, then $\ord(s)=\infty$. 
\item \label{srul:istr}
	If $K\subset{\cal H}$, $s\in{\cal S}$, and $s\le h$ 
	for all $h\in K$, then $\dim(s)\le d-\mathrm{card}(K)$.
	If $\dim(s)=d-\mathrm{card}(K)$, then $s\in{\cal T}$.
\item \label{srul:fac}
        $\ord(s)\ge m(s)$ for all $s\in{\cal S}$ and $m\in {\cal M}$.
\item \label{srul:down}
        ${\cal M}$ is downward closed with respect to $\leq$.\footnote{ 
        If $m_1(h)\leq m(h)$ for all $h\in{\cal H}$ and some $m\in {\cal M}$, then 
        $m_1\in {\cal M}$.} 
\item \label{srul:mono}        
        $\ord(s) - m(s)$ is weakly monotonously decreasing on $\cal S$, for all 
        $m\in {\cal M}$.\footnote{Note that if $\ord(s) - m(s)$ is weakly 
        monotonously decreasing for some $m\in {\cal M}$, then $\ord(s) - m_1(s)$ is 
        automatically weakly monotonously decreasing for any $m_1\leq m$; indeed, 
        $\ord(s) - m_1(s) = \ord(s) - m(s) +(m(s) - m_1(s))$ and $m(s) - m_1(s)$ is given 
        by its values on $h\in{\cal H}$.}
\item \label{srul:new}
        If ${\cal K}$ is a set of jibs such that $\sum_{h\in {\cal K}}m(h)\ge 1$ for some 
        $m\in{\cal M}$, then, for any node $s$ in ${\cal S}$ that lies below all nodes in 
        ${\cal K}$, there exists exactly one node $t\geq s$ in ${\cal S}$ which lies below all 
        nodes in ${\cal K}$ of dimension $d-\mathrm{card}({\cal K})$.\footnote{Note that 
        any node in this set has dimension $\le d-\mathrm{card}({\cal K})$ by Rule~\ref{rul:scen}, Issue~\ref{srul:istr}.}
\end{enumerate}
\end{rul}

A scenario  is {\em tight} if its order function is constant equal to~1 on whole $\cal S$. It is {\em resolved} if $\cal S$ is empty. A \factor{} $m$ is called {\em complete} if $\ord(s)=m(s)$ for all $s\in{\cal S}$. We then also say that $C$ is a {\em monomial} scenario.

%-------------------------------------------------

\paragraph{Transform of Boards.}

Mephisto may change the board by two types of modifications,
called {\em refinement} and {\em blowup}. In both cases, he has to give 
a new board $\Gamma'$ with new dimension labels.
On the level of nodes, Mephisto provides an embedding $i:\Gamma\to\Gamma'$, 
and a retract $u:\Gamma'\to\Gamma$, such that 
$u\circ i = \mathrm{id}_\Gamma$.\footnote{The map $i$ sends a stratum to a dense 
open subset, the inverse image $u^{-1}(s)$ consists of strata whose 
union is $s$. A refinement corresponds to a refinement of a stratification in the 
classical sense (strata are replaced by unions of strata), a blowup 
corresponds to a stratification such that all inverse images of strata
are unions of strata.}
The following rule must be fulfilled.

\begin{rul} \label{rul:trans}
The joint rules for both operations (refinement and blowup) are:
\begin{enumerate}
\item $i(s)$ is the unique maximal element 
	of $u^{-1}(s)$, for any $s\in\Gamma$.
\item $i(s)<i(t)$ if and only if $s<t$, for any $s,t\in\Gamma$.
\item $u$ is weakly monotonously increasing.
\end{enumerate}
Here is the additional rule for refinements.
\begin{enumerate} \addtocounter{enumi}{3}
\item $\dim(i(s))=\dim(s)$ for $s\in\Gamma$.
\end{enumerate}\goodbreak
Here are the additional rules for blowups.
\begin{enumerate} \addtocounter{enumi}{4}
\item Every blowup has a unique center $\cc\in\Gamma$, specified
	by Dido.
\item \label{srul:exdim}
	If $s\in\Gamma$ does not lie below $\cc$, then $\dim(i(s))=\dim(s)$.
\item \label{srul:indim}
	If $s\le \cc$, then $\dim(i(s))=\dim(s)+\dim(\Gamma)-1-\dim(\cc)$. \end{enumerate}
\end{rul}

%-------------------------------------------------

\paragraph{Transform of Scenarios.} 

Scenarios evolve under the moves of the players. If some quest remains open after a round of the game, then a scenario $C$ and its successor $C'$ -- chosen by Mephisto on a board $\Gamma$, respectively its refinement or blowup board $\Gamma'$ -- are not independent, but have to fulfill the following transformation rules.\footnote{We may think of two consecutive scenarios of a blowup move as being vertically related. There is also a horizontal relationship between scenarios on the same board; it is evoked by Dido when she issues a call and thus creates a subordinate quest to which Mephisto responds by a scenario. The horizontal relation lasts until one of the two quests becomes invalid (either because Mephisto loses it or because Dido gives it up). Deliberate divorce between horizontally related quests is not allowed in the game.} 

\begin{rul} \label{rul:refblow}
Let the board $\Gamma'$ be obtained from the board $\Gamma$ by a refinement or a blowup, with an embedding $i:\Gamma\to\Gamma'$ and a retract 
$u:\Gamma'\to\Gamma$. The rules for both, refinements and blowups of scenarios, are:

\begin{enumerate}
\item The dimension $d$ and the bound $B$ remain unchanged.
\item \label{srul:adm}
	For any $s\in\Gamma$, we have $i(s)\in{\cal T}'$ if and only if
	$s\in{\cal T}$.\footnote{There may be new transversal nodes in $\cal T'$ which are
        not in the image of $i$.}
\end{enumerate}
Here are the additional rules for refinements.
\begin{enumerate} \addtocounter{enumi}{2}
\item ${\cal S}'=u^{-1}({\cal S})$.
\item \label{srul:same}
	$\ord(s')=\ord(u(s'))$ for any $s'\in{\cal S}'$.
\item ${\cal H}'=i({\cal H})$. 
\item $m'\in{\cal M}$ if $m'(h')=m(u(h')$ for some $m\in{\cal M}$ and all $h'\in{\cal H'}$.
\end{enumerate}

Here are the additional rules for blowups. Let $\cc\in \Gamma$ be the center of the blowup, and $e:=i(\cc)$.\footnote {The node $e$ has dimension $n-1$, by Rule~\ref{rul:trans}, Issue~\ref{srul:indim}.}

\begin{enumerate} \addtocounter{enumi}{6}
\item \label{srul:center}
	$\cc\in{\cal T}$, and either $s\in{\cal S}$ or $s$ is remote from ${\cal S}$.
	Any node which satisfies these two conditions is called {\em admissible}.
\item ${\cal S}'\subseteq u^{-1}({\cal S})$.
\item If $\ord(\cc)<2$ or $\cc\not\in{\cal S}$, then $e\not\in{\cal S}'$. 
	Otherwise, $\ord(e)=\ord(\cc)-1$. 
\item For any $s'\in{\cal S}'$ which is not $\le e$, 
	we have $\ord(s')=\ord(u(s'))$.
\item \label{srul:tight}
	If $C$ is tight, then so is $C'$.
\item \label{srul:hand}
        ${\cal H}'=i({\cal H})\cup\{e\}$.
\item \label{srul:btr}
	If $K\subseteq{\cal H}$ with $\cc\le h$ for all $h\in K$ and
        $\dim(\cc)=d-\mathrm{card}(K)$, then there exists
        no node $s'$ in ${\cal S}'$ such that $s'\le i(h)$ for all $h\in K$. 
\item \label{srul:fact}
        $m'\in{\cal M'}$ if $m'(e)\leq \ord(\cc)-1$, provided $\cc\in{\cal S}$; otherwise
	$m'(e)=0$ -- 
	and there exists an $m\in {\cal M}$ so that
        $m'(i(h))= m(h)$ for all $h \in{\cal H}$.\footnote{The 
        transformation rule for $m$ corresponds to the way the exceptional factor of an ideal 
        transforms under blowup, cf. the combinatorial handicap in 
        \cite{Encinas_Hauser:02}.} 
\item \label{srul:complete} 
        If $m$ is a complete \factor{} of $C$, then $m'$ defined by $m'(e)= \ord(\cc)-1$ and  
        $m'(i(h))= m(h)$ for $h \in{\cal H}$ is a complete \factor{} of $C'$.\footnote{So 
        Mephisto has in this case no choice of how to choose $C'$.} 
\end{enumerate}
\end{rul}

We call $C'$ a {\em transform} of $C$ under a refinement, respectively blowup move. It is in general not unique, so that Mephisto has some freedom of how to choose its items.\footnote{But as $C'$ has to be again a scenario, Mephisto has to choose for a blowup transform of a scenario the order function so that 
Rule~\ref{rul:scen}, Issues~\ref{srul:fac} and \ref{srul:mono}, is satisfied. Observe that Mephisto has no choice of how to choose $\cal H'$ and $\cal M'$.}

%-------------------------------------------------

\paragraph{Quests.}

A {\em quest} is the collection of all scenarios which can be obtained from an initial scenario by blowup moves. It is thus a tree, where the directed edges connect scenarios which are related by a blowup. According to Dido's flexibility in choosing the center of the blowup, and Mephisto's flexibility in providing transformed scenarios, the tree may ramify considerably. However, playing a quest (i.e., applying concrete blowup moves) yields a sequence of transforms of the initial scenario which corresponds to a specific path in this tree. If a transform of a scenario reaches a final shape, the quest is won and discarded from the game (provided that it is not the main quest).

The overall game is constituted by several and interrelated quests, created at different moments (together with their initial scenario) and with possibly different life times. We say that a quest is {\em open} as long as it forms part of the game. When it is {\em closed}, it will be discarded from the game. The game starts with a single quest, the main quest, and an initial scenario thereof, which is provided by the Umpire. In the course of the game, other quests may open and close. The game ends when the main quest is won.

%-------------------------------------------------

\paragraph{Responses.}

Whenever a quest is open, Mephisto has to respond to it by a scenario when it is his turn to move. The scenario has to fulfill certain properties according to the type of the quest and the stage of the game. These are specified in the description below of the various quests and in the commutativity rules for blowup moves.

%-------------------------------------------------

\paragraph{Calls.} 

Aside of playing a blowup, Dido may also open at any time a new quest -- we also say: she {\em places a call}. This move introduces a new quest, considered as being {\em subordinate} to a quest specified by Dido; it will be part of the game until it is closed or deleted. Its scenarios are related to the scenarios of the superordinate quest by certain rules depending on the type of the call.  

The calls produce two types of quests, with different objectives: the first type are the relaxation and descent quest (for which Mephisto has some freedom on how to respond to them by scenarios), the second type are {\em one way} quests (for which the response scenario is uniquely determined by the superordinate quest). The latter calls are thus just commands without choice; they are needed for Dido in order to be able to add and factorize scenarios.  It is irrelevant whether Mephisto or the Umpire provides the corresponding response. The one way quests are the transversality and the \factorization{} quest.

Under a blowup move, a quest and its subordinate quest -- created by a call prior to the blowup -- preserve their relation defined by the type of the call (i.e., the respective scenarios have to obey the corresponding rules), see the subsection {\em Commutativity Relations.}\footnote{But it may happen that the subordinate quest closes under the blowup move and is thus discarded from the game.}

%-------------------------------------------------

\paragraph{Relaxation quest.}
Let $\mathfrak{Q}$ be a quest, and assume that ${\cal J}\subseteq{\cal H}$
is a chosen set of jibs of a scenario $C$ of $\mathfrak{Q}$ with board $\Gamma$. Then Dido may issue the call: ``Release ${\cal J}$!'' This creates a {\em relaxation quest}.

\begin{rul} \label{rul:relax}
A scenario $C_1$ on $\Gamma$ is a response scenario for the relaxation quest that releases ${\cal J}$ if the following rules hold.
\begin{enumerate}
\item
        The dimension $d$ and the bound $B$ are the same for $C$ and $C_1$.
\item
        ${\cal S}_1={\cal S}$ and $\ord_1 =\ord$.
\item
        ${\cal H}_1={\cal H}\setminus {\cal J}$.
\item	\label{srul:relax,T}	
	${\cal T}\subseteq{\cal T}_1$. 
	If $z\le h$ for all $h\in{\cal J}$, and $z\in{\cal T}_1$, then $z\in{\cal T}$.
	If $z$ is remote from ${\cal J}$, and $z\in{\cal T}_1$, 
	then $z\in{\cal T}$.
\item
        ${\cal M}_1$ consists of the restrictions to ${\cal H}_1$ of the \factors{}
        $m\in{\cal M}$.
\end{enumerate}
\end{rul}

%-------------------------------------------------

\paragraph{Descent quest.}

Let $\mathfrak{Q}$ be a quest, and assume that $C$ is a tight scenario of
$\mathfrak{Q}$. By Rule~\ref{rul:refblow}, Issues~\ref{srul:same} and \ref{srul:tight}, 
the future responses to $\mathfrak{Q}$ will again be tight,
so we may call the quest itself tight.  The tightness property of
a quest may be acquired during the game by some blowup, but once a quest
is tight it will stay so as long as it is open.

Assume, additionally, that the handicap $\cal H$ is empty.
Then Dido may issue the call: ``Step down!'' 
This creates a {\em descent quest}.\footnote{In the geometric situation of varieties, 
a descent in dimension to a hypersurface of maximal contact may create a transversality 
problem with the exceptional divisor. If the hypersurface is not transversal, 
a subordinate resolution problem in smaller ambient dimension is formulated 
in order to separate the hypersurface from the exceptional components, see the notion 
of transversality ideal in \cite{Encinas_Hauser:02}. In the game, we handle 
this difficulty by only allowing a descent call for scenarios with empty handicap. 
After blowup, this assumption is no longer required, as the new jibs of the transformed scenario 
will automatically be transversal, by Rule~\ref{rul:refblow}, Issues~\ref{srul:adm} 
and \ref{srul:hand}. }  
The response scenarios are subject to the following rules. 

\begin{rul} \label{rul:down}
Let $C$ be a tight scenario for $\mathfrak{Q}$ on a board $\Gamma$. A scenario $C_1$ on a refinement $\Gamma'$ of $\Gamma$ is a response scenario for the descent quest if the following rules hold.
\begin{enumerate}
\item $d_1=d-1$. 
\item ${\cal S}_1={\cal S}$, ${\cal H}_1={\cal H}$, and 
        ${\cal T}_1={\cal T}$.\footnote{Recall that only at the moment of the call, the
         actual scenario $C$ of $\mathfrak{Q}$ is required to have empty handicap.}
\end{enumerate}
\end{rul}

%-------------------------------------------------

\paragraph{One way quests.}

The next two quests are one way quests, the response scenario provided by Mephisto -- or the Umpire -- is always uniquely determined. It only depends on the superordinate scenario and the type of the quest. 

%-------------------------------------------------

\paragraph{Transversality quest.}

Let $\mathfrak{Q}$ be a quest, with scenario $C$ on a board $\Gamma$, 
and let ${\cal K}\subseteq {\cal H}$ be a set of jibs. 
Then Dido may construct the {\em transversality quest} $\mathfrak{Q}_1$ with respect to ${\cal K}$.\footnote{In contrast to \cite{Encinas_Hauser:02} where products of ideals are taken to deal with the transversality problem, we axiomatize here a transversality ideal which is given by the restriction of the singular ideal to the intersection of certain exceptional hypersurfaces. This corresponds to the sum of ideals, respectively gallimaufries.} Its response scenario is given on $\Gamma$ by 

\begin{rul} \label{rul:jib}
\begin{enumerate}
\item $d_1=d$, $B_1=B$.
\item ${\cal H}_1={\cal H}$.
\item ${\cal S}_1=\bigcap_{h\in {\cal K}}\ \{s\in {\cal S}\mid s\le h\}$.\footnote{In case that ${\cal K}=\emptyset$, this signifies that ${\cal S}_1={\cal S}$.}
\item $\ord_1(s)=1$ for all $s\in{\cal S}_1$, except for ${\cal K}=\emptyset$ where $\ord_1=\ord$.
\item \label{srul:jib,T}
	${\cal T}_1={\cal T}$.
\item ${\cal M}_1=\{0\}$, except for ${\cal K}=\emptyset$ where ${\cal M}_1={\cal M}$, 
        and for ${\cal K}=\{h\}$, where ${\cal M}_1$ contains 
        the function mapping $h$ to $1$ provided that it belongs to ${\cal M}$. 
\end{enumerate}
\end{rul}

%-------------------------------------------------

\paragraph{Quotient quest.}

Let $\mathfrak{Q}$ be a quest, with scenario $C$ on a board $\Gamma$. Let  $m$ be a \factor{} of $C$ and let $q>0$ be a positive rational number, the {\em scale}. Both are chosen by Dido. The response scenario $C_1$ to the {\em quotient quest} is called the {\em $q$-quotient} of $C$ with respect to $m$.\footnote{\Factors{} and quotients axiomatize the combinatorial handicap of a mobile from \cite{Encinas_Hauser:02}; they enable Dido to split off from an ideal the exceptional monomial factor. The $q$-quotient corresponds to the remaining non-monomial factor, and the scale $q$ adjusts its control.} It is given on $\Gamma$ by 

\begin{rul} \label{rul:quot} 
\begin{enumerate}
\item $d_1=d$, $B_1\cdot \Z=\Z\cap \frac{B}{q}\cdot\Z$. 
\item ${\cal H}_1={\cal H}$. 
\item ${\cal S}_1=\{s\in{\cal S}\mid \ord(s)-m(s)\ge q \}$. 
\item \label{srul:ord}
        $\ord_1(s)=\min\ \{\ord(s),\frac{1}{q}\cdot (\ord(s)-m(s))\}$.
\item ${\cal T}_1={\cal T}$.
\item ${\cal M}_1=\{ \frac{1}{q}\cdot (f(s)-m(s)) \mid f\in{\cal M}\}$.
\end{enumerate}
\end{rul}

If $q>\ord(s)-m(s)$ for all $s\in{\cal S}$, the singular set ${\cal S}_1$ is empty, the quotient quest resolved and discarded. So Dido will preferably choose smaller values of $q$. If $q$ equals the maximal value of $\ord(s)-m(s)$ on $S$, the response scenario is tight.
 
%-------------------------------------------------

\paragraph{Commutativity Relations.}

New quests are created by calls. This establishes an asymmetric relation between the superordinate and the subordinate quest. The relation is specified by the type of the call. Along a sequence of blowups, the relation has to be maintained by the players. We specify this persistence by listing for each call the required properties between scenario and subordinate scenario.

We consider blowups in centers $\cc$ contained in the singular and transversal sets of a scenario $C$ of a quest $\mathfrak Q$, and of a scenario $C_1$ of some subordinate quest $\mathfrak Q_1$. Both scenarios are defined on the same board $\Gamma$ and may undergo a (not necessarily unique) blowup transform with the same center $\cc$. Let $e=i(\cc)$ be the exceptional node of the blowup, and let $i:\Gamma\to\Gamma'$ and $u:\Gamma'\to\Gamma$ be the associated inclusion and retract, see Rules~\ref{rul:trans} and \ref{rul:refblow}.

We denote by $C'$ a blowup transform of $C$ chosen by Mephisto, according to Rule~\ref{rul:refblow}. Let $(C')_1$ be Mephisto's response scenario to $C'$ within the quest $\mathfrak Q_1$. Then, in order to be an allowed response of Mephisto, $C'_1$ has to satisfy the following properties.

\begin{rul} \label{rul:comm}
\begin{enumerate}
\item \label{srul:relax}
        If $C_1$ is a relaxation scenario of $C$ with respect to
        ${\cal J}\subseteq{\cal H}$, then $(C')_1$ is a relaxation scenario of $C'$ with
        respect to ${\cal J'}=i({\cal J})$ {\em and} a transform
        $(C_1)'$ of $C_1$.
\item \label{srul:descent} 
        If $C_1$ is a descent scenario of $C$, then $(C')_1$ is a descent scenario of $C'$ 
        {\em and} a transform $(C_1)'$ of $C_1$.
\item \label{srul:jib}
        If $C_1$ is the transversality scenario of $C$ with respect to ${\cal K}\subseteq{\cal 
        H}$, then $(C')_1$ is the transversality scenario of $C'$ with respect to ${\cal 
        K}'=i({\cal K})$ {\em and} a transform $(C_1)'$ of $C_1$.
\item \label{srul:qtr}  
        If $C_1$ is the $q$-quotient of $C$ with respect to a \factor{} $m$, then $(C')_1$ is 
        the $q$-quotient of $C'$ with respect to $m'$ defined by $m'(e) = m(\cc)+q-1$ and  
        $m'(i(h))= m(h)$ for $h \in{\cal H}$, {\em and} $(C')_1$ is a transform $(C_1)'$ of 
        $C_1$.\footnote{Observe that $m(\cc)+q\leq 
        \ord(\cc)$ if $\cc$ belongs to the singular set ${\cal S}_1$ of $C_1$, by 
        Rule~\ref{rul:quot}, Issue~\ref{srul:ord}. If the center is not admissible for $C_1$, 
        say, if $\cc$ is not in ${\cal T}$ or $\cc$ is not in ${\cal S}_1$ and not remote from 
        ${\cal S}_1$, the quotient quest is discarded after the blowup.}
\end{enumerate}
\end{rul}

These relations are schematized by the following diagram.\footnote{The response scenarios of the call after a blowup have to be transforms of the response scenarios of the call before the blowup.} 
%-------------------------------------------------
\vskip .5cm 

\hskip 4cm call

\vskip .2cm

\hskip 3cm $C' \hskip .6cm \looparrowright \hskip .2cm (C')_1 =(C_1)'$

\vskip .4cm 

\hskip 1.7cm blowup \hskip 0.1cm $\downarrow  \hskip 2.5cm \downarrow$ \hskip 0.1cm blowup

\vskip .4cm 

\hskip 3cm $C \hskip .7cm \looparrowright \hskip 1.1 cm C_1$

\vskip .2cm

\hskip 4cm call

\vskip .5cm 
%-------------------------------------------------

Observe here that the transforms of a scenario are not uniquely prescribed; but Mephisto's responses have to be chosen so that commutativity holds.

%-------------------------------------------------

\paragraph{Playing the Game.}

The game starts with a board $\Gamma$ carrying a single scenario $C$, 
the \emph{initial} scenario. It is provided by the Umpire and represents 
the first scenario of the main quest. In the course of the game, 
this scenario will transform by blowups in a way which is 
governed by the moves of the players. The game is finished, or the main quest 
is \emph {won}, when the singular set of the transform of the initial 
scenario has become empty.  

The two protagonists Dido and Mephisto play their moves alternately. 
The moves affect and modify the actual scenarios of all open quests.

During the game, other quests may be opened by Dido. Each of them remains valid until 
it is won or given up. Dido gives up a quest 
by choosing a center that is not admissible for the scenario of the quest. In this case, there is no response scenario for the quest after the blowup,
according to Rule~\ref{rul:refblow}, Issue~\ref{srul:center}, and the quest is discarded from the game. 

A possible evolution of the game is depicted in Figure 3. The dots represent the various scenarios provided by Mephisto, blowups are drawn vertically, calls horizontally. The numbering indicates the sequence of rounds formed by Dido's move (blowup or call) and Mephisto's response (scenarios).\footnote {The arrows labelled with $5$ correspond all to a round given by a blowup move: it applies to all quests for which the center is admissible, and Mephisto's responses have to provide scenarios for all these quests. Observe that the transversality quest labelled with 6 already belongs to the next round.}

In each round of the game, all scenarios of open quests are built on the same board $\Gamma$. When a new quest is opened by Dido, Mephisto may refine in his response the underlying board. Similarly, under a blowup move of Dido, the board $\Gamma$ transforms into a board $\Gamma'$. In both cases, the new board will be the common underlying graph for the scenarios of Mephisto's responses. 

If the singular set of a scenario $C$ is empty, the respective quest 
is closed, and Mephisto has lost this quest. So assume that some quests 
are still open, in particular, the main quest. It is Dido's turn to move. She has three choices. 

\begin{enumerate}
\item She plays a blowup by nominating an admissible node $\cc$ in the actual board 
	$\Gamma$. The quests for which $\cc$ is not admissible are closed and will be 
	discarded, the others remain open.
\item She specifies one actual scenarios and calls a one way quest for 
        them. The call will be a transversality quest or a quotient quest. In case of a 
        transversality quest, a set of jibs has to be specified.
        In case of a quotient quest, the \factor{} and the scale have to be 
        specified.
\item She specifies an actual scenario and calls a relaxation or descent quest for it. 
	This opens a new quest, subordinate to the quest of the chosen scenario. The new 
	quest is then played simultaneously with the other, already open quests. 
\end{enumerate}

Now it is Mephisto's turn.  First, assume that Dido has played a blowup with node $\cc\in\Gamma$. In this case, Mephisto provides a blowup transform $\Gamma'$ of the actual board $\Gamma$, together with an embedding $i:\Gamma\to\Gamma'$ and a retract $u:\Gamma'\to\Gamma$. Then, he responds to all open quests by specifying appropriate scenarios on $\Gamma'$.  For each such quest, the response scenario $C'$ is related to the previous scenario $C$ of the quest by the the transformation laws of Rule~\ref{rul:refblow}. In addition, the responses of Mephisto need to respect the {\em Commutativity Relations} from Rule~\ref{rul:comm} above.

Second, assume that Dido has placed a one way call, say a transversality or a \factorization{} quest. In this case, Mephisto has no choice and he or the Umpire provide the uniquely determined scenario.\footnote{Recall here that the factor of an initial scenario is set equal to $0$.}

%-------------------------------------------------
%      PICTURE GAME
%-------------------------------------------------

\begin{figure}
\begin{center}
\includegraphics[width= 0.8\textwidth ] {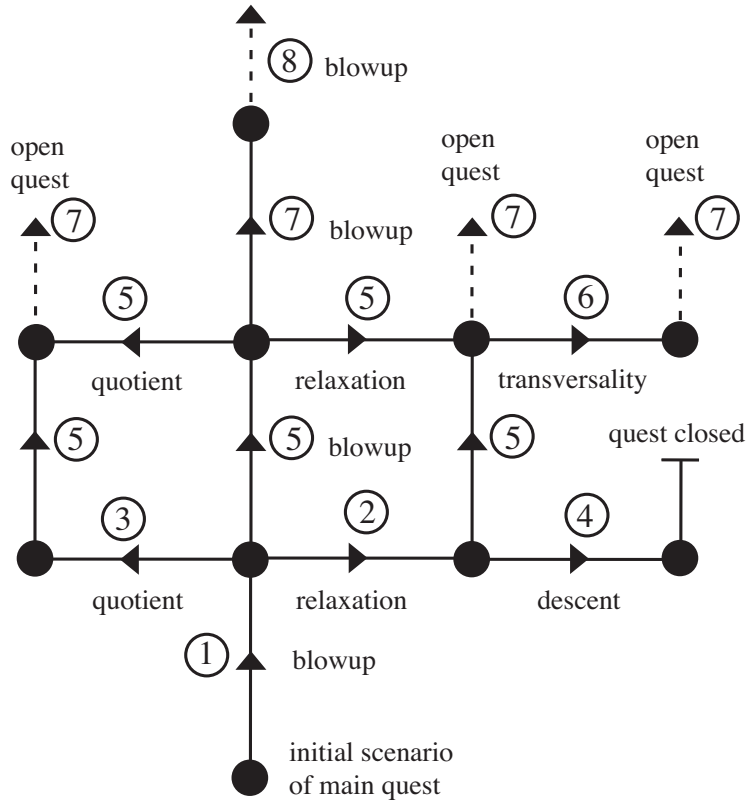}
\caption{ \label{fig:game} Playing Salmagundy.}
\end{center} \vskip -.4cm

\end{figure}
%-------------------------------------------------

Finally, assume that Dido has placed a relaxation or descent call, thus creating a new quest to which Mephisto has to respond. He provides first a refinement $\Gamma'$ of the actual board $\Gamma$, together with an embedding $i:\Gamma\to\Gamma'$ and a retract $u:\Gamma'\to\Gamma$. Then, he responds to all open quests by specifying the respective scenarios on $\Gamma'$. In his choice of scenarios, he has to respect three rules: 

\begin{enumerate}
\item The scenarios of all open quests except of the new one must be refinements of the 
        previous scenarios (i.e., those provided by Mephisto before the call) in the sense of 
        Rule~\ref{rul:trans}.\footnote{The initial scenario of the new quest has no 
        predecessor.}
\item The scenarios of the quests which were already open before the last blowup move of   
        Dido must be a transform of the respective scenario of the quest before the blowup as 
        indicated in Rule~\ref{rul:refblow}.\footnote{The initial scenario of the new quest 
        cannot be a transform. Its factor is $0$.}\goodbreak
 \item The responses of Mephisto need to respect the {\em Commutativity Relations} from 
        Rule~\ref{rul:comm}.
\end{enumerate}

The next round proceeds as before with Dido's move followed by Mephisto's 
response. The game continues like this until Mephisto loses, or forever. Dido cannot 
lose, because even if there is no admissible center of blowup to be chosen there are
infinitely many ways to open new quests. But Dido wants more: she wants to 
win. Her goal is to choose her moves so that, regardless of the responses of Mephisto, 
the resulting path of scenarios in the tree of the main quest leads from the initial scenario to a final scenario. Along the way, she may (and will) 
win subordinate quests which she has opened by her calls.\footnote{According to ancient mythology, Dido was a fugitive who arrived with nothing in the country of king Jarbas.
After asking for only so much land she could fit into the skin of a cow,
she cut the skin into small stripes and spanned a large area which later
became the mighty city of Carthage. So Dido is a clever opponent, 
dangerous to Mephisto.}

%-------------------------------------------------
%            GALLIMAUFRY
%-------------------------------------------------

\section{Gallimaufry}

In his proof for the resolution of singularities over fields of characteristic zero, Hironaka used a multiple induction between various resolution statements carrying on different types of resolution data \cite{Hironaka:64}, chap.÷ I, sec.÷ 2. Similarly, Abhyankar introduced several types of resolution data for proving resolution in small dimensions and positive characteristic \cite{Abhyankar:66}, see also the last section of Lipman's survey article \cite{Lipman:75}. Later on, many variants of resolution data have been proposed and used in the literature \cite{Villamayor:89, Villamayor:92, Bierstone_Milman:97, Encinas_Hauser:02, Wlodarczyk:05, Kawanoue:07, Kawanoue_Matsuki:08}. They all share the necessity of book-keeping the configuration of the exceptional divisors as well as certain numerical data like orders and Hilbert-Samuel functions of ideals.\footnote{Most of the notions in this section have been proposed, often in slightly different manner, by Hironaka and/or Villamayor. Our terminology does not intend to overrule their notions. But due to the various differences, a systematic new terminology was preferable.}

In this section we introduce still another type of resolution datum, called \emph{Gallimaufry}.\footnote{{\em Gallimaufry}:  A jumble, a hodgepodge.
{\em Galimafr\'ee}: From Old French {\em calimafree}, sauce made of mustard, ginger, and vinegar (Douglas Harper, Online Etymology Dictionary). A motley assortment of things (Thesaurus).} The structure of gallimaufries allows an induction in dimension that does not depend on any local choice of hypersurfaces. The global definition then permits a significant simplification of the induction argument.\footnote{Originally, local data had to be glued together via complicated equivalence relations in order to perform the descent in dimension \cite{Hironaka:77, Villamayor:89, Villamayor:92, Bierstone_Milman:97}. In \cite{Encinas_Hauser:02} a global resolution datum, called {\em mobile}, was defined without any gluing. The local resolution invariant could then be defined directly by a local surgery, and Hironaka's trick  showed its independence from any choices. In the present paper, the construction of mobiles is refined even further, combining it with ideas from \cite{Wlodarczyk:05, Kollar:07} and, most essentially, from \cite{Hironaka:07, Villamayor:07}. }

\begin{defn}
A {\em habitat} is a pair $W=(W,E)$ consisting of a connected algebraic manifold $W$ 
over a field of characteristic zero (i.e., a non-singular algebraic variety),
together with a finite set $E=\{E_1,\dots,E_r\}$ of non-singular hypersurfaces
forming a normal crossings divisor. A habitat restricts to open subsets $U$ of $W$ by taking $\{E_1\cap U,\dots,E_r\cap U\}$. 

A non-empty closed subvariety $Z\subset W$ is called {\em transversal} to $(W,E)$ if 
for every point $p\in Z$, there is a system
of regular local parameters such that $Z$ is defined by a subset of these
parameters and such that every hypersurface $E_i$ containing $p$ is
defined by one of these parameters. The set $E$ may be empty. In this case, a subvariety is transversal if and only if it is non-singular.\footnote{Geometry: One should think of the set $E$ as the collection of the exceptional divisors created by the blowups so far; they will be used later on to identify and separate the combinatorial portion of a resolution problem.}

\end{defn}

%-------------------------------------------------

\begin{prop}
The blowup of a habitat along a transversal subvariety, 
with the proper transforms of hypersurfaces in $E$ 
supplemented by the exceptional divisor,
is again a habitat.
\end{prop}

\begin{proof} 
Let $f:W'\to W$ be the blowup. Let $p'\in W'$, and let $p=f(p')$. Let $n:=\dim(W)$ and $m:=\dim(Z)$, where $Z$ is the blowup center. Assume that $u_1,\dots,u_n$ is a set of local parameters at $p$ such that $Z$ is locally defined by $u_1,\dots,u_m$, and each of the hypersurfaces in $E$ is locally defined by $u_i$ for some $i$, $1\le i \le n$. Then any point in $f^{-1}(p)$ is contained in some affine open chart of $W'$ with regular parameters $u_j,\frac{u_1}{u_j},\dots,\frac{u_m}{u_j},u_{m+1},\dots,u_m$ for some $j$, $1\le j\le m$, or translations of these functions e.g. $\frac{u_1}{u_j}-c$ where $c$ is a constant. This system of regular local parameters has the required properties.
\end{proof}

%-------------------------------------------------

\begin{defn}
Let $(W,E)$ be a habitat. We say that an ideal sheaf $I\subset\OO_W$ of the structure sheaf  $\OO_W$ of regular functions on $W$ is {\em principal monomial} if $I$ is a tensor product of the invertible ideal sheaves of $E_i$, for $E_i\in E$.  
A {\em resolution} of $I$ is a finite sequence of blowups of the habitat 
along transversal centers  
$W_r \rightarrow\dots\rightarrow W_1  \rightarrow W$ 
such that the pullback of $I$ in $\OO_{W_r}$ is principal monomial.\footnote{If $I$ is the ideal sheaf of a hypersurface $X$, then a resolution of $I$
is a resolution of $X$, in the sense that the inverse image of $X$ is
the union of non-singular hypersurfaces forming a normal crossings divisor.}
\end{defn}

%-------------------------------------------------

\subsection{Ideals and Algebras.} 

In the next few pages, we define the relevant algebraic objects of our study. As the role of the set of hypersurfaces $E$ is only to keep track of the transversality of the possible centers of blowup with the components of the exceptional divisor, we will not mention $E$ until explicitly needed.

%-------------------------------------------------

\begin{defn}
An {\em ideal with control} $(I,c)$ on a manifold $W$ is an ideal $I\subset\OO_W$
together with an integer $c>0$, the {\em control}.\footnote{We follow here and in the sequel Hironaka and his notion of {\em idealistic exponent}  \cite{Hironaka:77}.} The {\em singular locus} of $(I,c)$ is the set of all points $p\in W$ such that the order of $I$ at $p$ is at least $c$. The {\em sum} of two ideals with control $(I_1,c_1)$ and $(I_2,c_2)$ on the same manifold is defined as the ideal $I_1^{c_2}+I_2^{c_1}$ with control $c_1c_2$. We say that $(I_1,c_1)$ and $(I_2,c_2)$ are {\em equivalent} if there exists a positive integer $k$ such that $I_1^{kc_2}=I_2^{kc_1}$. 
\end{defn}

\begin{defn}
Let $Z\subset W$ be a transversal subvariety contained in the singular locus
of $(I,c)$, and let $f:W'\to W$ be the blowup of $W$ along $Z$. Any element in
the pullback $f^\ast(I)$ is divisible by the $c$-th power of a local
generator of the exceptional divisor. Therefore $f^\ast(I)$ can be 
written as a product $M^c\cdot I'$, where $M$ is the ideal of the exceptional divisor $D=f^{-1}(Z)$ and $I'$ is an ideal sheaf on $W'$. The pair $(I',c')=(I',c)$ is called the {\em transform} of $(I,c)$ under $f$, or the {\em controlled transform} of $I$.

A {\em resolution} of an ideal $I$ with control $c$ is a finite sequence of blowups
with transversal centers contained in the singular locus such that the last singular locus is empty.\footnote{The existence of resolutions of ideals can easily be reduced to the existence of resolutions of controlled ideals: one just needs to resolve the
ideal with control equal to 1. Then the last transform is
$(\OO_{W_r},1)$, and since the pullback of $I$ can always be written
as a product of its controlled transform and a monomial ideal, the last
pullback is a monomial.}
\end{defn}

%-------------------------------------------------

\begin{prop}
Equivalent ideals with control have the same singular locus, equivalent
transforms and the same resolutions. \end{prop}

%\begin{proof} Straightforward. \end{proof}

%-------------------------------------------------

\begin{prop}
The sum operation is commutative and associative.
The singular locus is the intersection of the singular loci of the summands.
If a transversal subvariety is contained in the singular locus of the sum,
then the transform of the sum is the sum of the transforms of the summands.
\end{prop}

%\begin{proof} Straightforward. \end{proof}

%-------------------------------------------------

\begin{defn} \label{def:rees}
A {\em Rees algebra} $A$ on a manifold $W$ is a coherent sheaf of locally finitely generated subalgebras of $\OO_W[T]$, where $T$ is a variable. 
We will write $A=\oplus A_i$ for the decomposition into homogeneous components. The $A_i$ are ideals in $\OO_W$, and will be considered with control $i$. The {\em singular locus} of a Rees algebra is the set of all points
where $A_i$ has order at least $i$, for each $i\geq 0$. 

If $Z$ is a transversal subvariety of $(W,E)$ inside the singular locus of $A$, the 
{\em transform} of $A$ under the blowup of $W$ along $Z$ is the algebra generated by all controlled transforms of $A_i$, for all $i>0$.

A {\em resolution} of a Rees algebra is a finite composition of blowups along
transversal centers inside the singular locus such that the
last singular locus is empty.

Let $A$ and $B$ be two Rees algebras. We write $A\subseteq B$ if and only if $A_{i}\subseteq B_{i}$ for each $i>0$; we say that $A$ and $B$ are {\em equivalent} if and only if there exists a positive integer $k$ such that $A_{ki}=B_{ki}$ for all $i>0$; 
% [Shall we better take the integral closure of Rees algebras?
% But then we would depend on Zariski-Samuel and others.]
the {\em sum} of $A$ and $B$ is the smallest Rees algebra containing both.
\end{defn}

%-------------------------------------------------

\begin{prop}
Equivalent Rees algebras have the same singular locus, equi\-valent
transforms under blowup, and the same resolutions.
\end{prop}

%\begin{proof} Straightforward. \end{proof}

%-------------------------------------------------

\begin{defn}
If $K$ is a finite set of positive integers and $B_k$ is an $\OO_W$-ideal for all $k\in K$, we define the Rees algebra $A=\oplus_{i=0}^\infty A_i$ {\em generated} by $B_k$, $k\in K$, by taking for $A_i$ the ideal generated by all products $B_{k_1}B_{k_2}\dots B_{k_m}$ with $k_1+\dots+k_m=i$. Thus $A$ is the smallest Rees algebra containing $B_k$ in degree $k$, for all $k\in K$. If $K$ is a finite set such that $A=\oplus_{i=0}^\infty A_i$ is the Rees algebra generated by $A_k$, $k\in K$, then we say that $K$ is a {\em set of generating degrees}.\footnote{Any Rees algebra $A$ has a set of generating degrees: Take a finite cover of $W$ by affine open subsets such that $A$ is finitely generated on each of these. If $N$ is the maximal degree of all these local generators, 
then $\{1,\dots,N\}$ is a set of generating degrees for $A$.}
\end{defn}
%-------------------------------------------------

Both Rees algebras and ideals with control are algebraic realizations
of ``resolution problems''. It is possible to go back and forth between
them.

%-------------------------------------------------

\begin{defn}
Let $A$ be a Rees algebra and let $I$ be an ideal with control $c$.
We say that $A$ and $(I,c)$ are {\em associated} if and only if there exists
a positive integer $k$ such that $A_{kci}=I^{ki}$ for all $i>0$. 
%[Does this make sense? Don't we only want inclusions $A_i \subset I^{j_i}$ and $I^j\subset A_{i_j}$? ]
\end{defn}

%-------------------------------------------------

\begin{prop}
Every Rees algebra is associated to an ideal with control, which is
unique up to equivalence. 
%[I don't believe this. Proof explicited below] 
Conversely, every ideal with control is
associated to a Rees algebra, which is unique up to equivalence.

Associated ideals with control and Rees algebras have the same singular locus, and associated transforms under blowup. The sums of associated pairs of ideals with control and Rees algebras are again associated.
\end{prop}

\begin{proof} 

The Rees algebra associated to the ideal $I$ with control $c$ is the algebra
generated by $I$ in degree $c$. Conversely, let $A$ be a Rees algebra with
generating set $K$, and let $n$ be an integer which is divisible by all integers
in $K$. Then the ideal $A_n$ with control $n$ is associated to $A$.
\end{proof}

%-------------------------------------------------

\begin{defn}
Denote by $\Delta$ the operator that takes an ideal $I$ in $\OO_W$
and produces the ideal $\Delta(I)$ generated by $I$ and all first order partial derivatives of sections of $I$.
For any Rees algebra $A=\oplus_{i=0}^\infty A_i$ with set of generating degrees $K$, the algebra $\widetilde A$ generated by 
$\Delta^j(A_k)$ in degree $k-j$ for $k\in K$ and $j<k$ is called
the {\em differential closure} of $A$.\footnote{The definition
does not depend on the choice of the set $K$ of generating degrees.}

A {\em Villamayor algebra} is a differentially closed Rees algebra $A=\widetilde A$.  Equivalently, it suffices to require that $\Delta(A_{i+1})\subseteq A_i$ for all $i\ge 0$.\footnote{The differential closure is the smallest Villamayor algebra containing $A$. Note that if $A$ is a Villamayor algebra, then we have $A_{i+1}\subseteq A_i$ for each $i>0$. This is a consequence of $I\subseteq \Delta(I)$.}   
A set $K$ such that $\Delta^j(A_k)$ generates $A$ in degree $k-j$ for $k\in K$ and $j<k$ is called a set of generating degrees of $A$ as a Villamayor algebra. 
\end{defn}

%-------------------------------------------------

\begin{defn}
Let $A$ be a Rees algebra. Then the {\em interior} $A^\circ$ of $A$ is defined as
the Villamayor algebra generated by all Villamayor algebras $B$ contained in $A$. 
%[Does this make sense? Is it clear that this is again an algebra?]
\end{defn}

%-------------------------------------------------

\begin{rem}
The interior algebra can be constructed by induction on the degree $i$.
We set $A_0^\circ=A_0=\OO_W$ and $A_1^\circ =A_1$. For $i>1$, the sections of $A_i^\circ$ are the sections $f$ of $A_i$ that satisfy $\Delta(f)\subseteq A_{i-1}^\circ$. 
For obtaining a set of generating degrees of $A^\circ$, it suffices to do the construction
up to the largest degree of a set of generating degrees of $A$.
\end{rem}

An essential step of Hironaka's resolution proof is induction on the
dimension. We will prepare such an induction by defining restrictions of Rees algebras
to smooth, locally closed subvarieties. Even though the restrictions may be local in nature, the induction argument will remain global: restriction plays only an auxiliary role.

%-------------------------------------------------

\begin{defn}
Let $(W,E)$ be a habitat and let $V$ be a transversal subvariety of $W$. The algebra $A(V)$ is generated in degree~1 by the ideal defining $V$ in $W$. If $A$ is a Villamayor algebra on $W$ such that $A(V)\subseteq A$, then the restriction $A|_V$ is defined as the Villamayor algebra on $V$ with $i$-th component the image of $A_i$ in the quotient ring $\OO_V$ of $\OO_W$.
\end{defn}

%-------------------------------------------------

\begin{defn}
A {\em gallimaufry of dimension} $d$ in a habitat $(W,E)$ is a pair
$\GG= (A,d)$, where $A$ is a Villamayor algebra on $W$ and $d$ is an integer  $0\le d\le\dim(W)$, such that every point $p$ in the 
singular locus of $A$ has a neighborhood $U$ and a transversal subvariety
$V\subseteq U$ of dimension $d$ not contained in any hypersurface of $E$ 
so that $A(V)\subseteq A$ holds in $U$.
Such a local transversal subvariety of $W$ is called a {\em zoom} for $(A,d)$ at $p$.\footnote{As a first example of a gallimaufry, take any Villamayor algebra $A$ and set $d=\dim(W)$. Then $W$ itself is a zoom for any point $p\in W$. More generally, zooms mimic the notion of hypersurfaces of maximal contact, and the restriction of gallimaufries to zooms captures the passage from ideals to coefficient ideals.}

The {\em singular locus}, the {\em transform} under blowup and the {\em resolution} of a gallimaufry $(A,d)$ are defined as the respective items of $A$.
\end{defn}

%-------------------------------------------------

\subsection{Transforms.} 

We next describe the behaviour of the various algebraic items under blowup. We start with Rees algebras.

\begin{lem} \label{lem:bold}
The singular locus of $A(V)$ is equal to $V$. If we blow up a proper
subvariety $Z\subset V$, the transform of $A(V)$ is equal to $A(V')$, where $V'$ is the
strict transform of $V$.  If we blow up $Z=V$, then the transform is the trivial algebra generated by $\OO_W$ in degree $1$.
\end{lem}

\begin{proof}
Clear.
\end{proof}
%-------------------------------------------------

\begin{lem} \label{lem:rescom}
Let $A$ contain $A(V)$. The singular loci of $A$ and $A|_V$ coincide. If $Z\subset V$ is
a transversal subvariety contained in the singular locus, then
the transform of the restriction $A|_V$ under the blowup of $V$ along $Z$ is equal to the restriction of the transform of $A$ under the blowup of $W$ along $Z$ to the strict transform $V'$ of $V$.
\end{lem}

\begin{proof}
We can choose regular local parameters such that $V$ is given by a subset
of these, and identify $A$ with the result of the first description
of the extension of $A|_V$. The computation is straightforward.
\end{proof}

%-------------------------------------------------

\begin{lem}[Giraud] \label{lem:im} 
Let $I$ be an ideal with control $c$ on $W$ and let $Z$ be a transversal
subvariety of $(W,E)$ contained in the singular locus of $(I,c)$. 
Let $I'$ be the controlled transform of $(I,c)$ under the blowup of $W$ along $Z$. 
Then the controlled transform of the ideal $\Delta(I)$ with control $c-1$
is contained in $\Delta(I')$.
\end{lem}

\begin{proof}
Let $f:W'\to W$ be the blowup.
Let $p'\in W'$ be a point on the exceptional divisor $D$, and set $p=f(p')$.
Let $x\in\OO_{W',p'}$ be a local equation of $D$.
Let $\eta\in\mathrm{Der}(\OO_{W,p})$ be a derivation.
Then $f^{\ast}(\eta)$ has at most a simple pole along $D$, so 
$\eta':=xf^{\ast}(\eta)$ is a derivation 
in $\mathrm{Der}(\OO_{W',p'})$.
It is tangential to $D$. If we replace $x$ by a suitable analytic 
generator of $D$, we may assume $\eta'(x)=0$.

The controlled transform of $(\Delta(I),c-1)$ is generated by elements
of the form $x^{1-c}\cdot f^\ast(\eta(a))$, with $a\in I$. We compute
\[ \frac{f^\ast(\eta(a))}{x^{c-1}}=
  \frac{f^\ast(\eta)(f^\ast(a))}{x^{c-1}}=
  \frac{\eta'(f^\ast(a))}{x^c}=
  \eta'\left(\frac{f^\ast(a)}{x^c}\right) , \]
hence $f^\ast(\eta(a))\in\Delta(I')$. Since the elements of form
$\eta(a)$ generate $I$ at $p$, it follows that 
$f^\ast(\Delta(I))\OC_{W',p'}\subseteq\Delta(I')\OC_{W',p'}$.
Completion is a faithfully exact functor,
hence the statement is also true for the ideals in the local rings.
\end{proof}

%-------------------------------------------------

\begin{lem}[Villamayor] \label{lem:later} 
Let $A$ be a Rees algebra and let $B$ be its differential closure.
Then the singular loci of $A$ and $B$ are equal. If $Z$ is a transversal
subvariety contained in the singular set, and $A'$ and $B'$ are the
transforms of $A$ and $B$, then $A'$ and $B'$ have the same differential
closure.
\end{lem}

\begin{proof} 
The singular locus is equal to the intersection of all zero sets of
all $(i-1)$-st derivatives of $A_i$, for $i>0$. But these are all in $B_1$,
hence the singular locus of $B$ is contained in the singular locus of $A$.
The other direction is obvious, because $A$ is a subset of $B$.

Let $f:W'\to W$ be the blowup along $Z$.
The differential closure of $A'$ is clearly contained in the differential
closure of $B'$.
For the converse, it suffices to show that $B'$ is contained in the
differential closure of $A'$. Let us denote the differential closure
of $A'$ by $C'$. 

We prove by induction over $j$ that for all $i>0$, the ideal of
the controlled transform of $\Delta^j(A_{i+j})$ with control $i$ 
is contained in $C'_i$. 
Since $B_i$ is generated by the ideals 
$\Delta^j(A_{i+j})$, $j\ge 0$, it will follow that $B'\subseteq C'$.

The case $j=0$ is obvious. Assume $j>0$. Let $i>0$, and set 
$I=\Delta^{j-1}(A_{i+j})$, with control $i+1$. By induction hypothesis, 
the transform $I'$ of $I$ with control $i+1$ is contained in $C'_{i+1}$.
By Lemma~\ref{lem:im}, the controlled transform of $\Delta(I)$ with control $i$ 
is contained
in $\Delta(C'_{i+1})$. And $\Delta(C'_{i+1})\subseteq C'_i$, because
$C'$ is differentially closed. 
\end{proof}

%-------------------------------------------------

\begin{prop} \label{prop:trf}
If $(A,d)$ is a gallimaufry, and $Z$ is a transversal subvariety of $(W,E)$
in the singular locus of $(A,d)$, and $\widetilde A'=\widetilde {(A')}$ is the differential closure of
the transform $A'$ of $A$ under the blowup of $W$ along $Z$, then $(\widetilde A',d)$ is again a gallimaufry. 
\end{prop}

\begin{proof}
Let $f:W'\to W$ be the blowup. Let $p'\in W'$. Then it is clear that
the strict transform of any zoom at $f(p')$ is a zoom at $p'$.
\end{proof}

%-------------------------------------------------

\subsection{Restriction.} 

An important property of Villamayor algebras is their ``stability'' under restriction. This is made precise in the following statement.

\begin{thm} \label{thm:aus}
Let $(W,E)$ be a habitat and let $V\subset W$ be a transversal subvariety.
Then the restriction operator from Villamayor algebras on $W$ containing $A(V)$ to Villamayor algebras on $V$ is bijective.

The inverse operator -- called extension -- can be constructed in two ways. 
For the first construction, we assume that we have a left inverse $\beta:W\to V$ 
to the inclusion map $i:V\to W$. Then the extension of $A$ from $V$ to $W$ is equal to
$\beta^\ast(A)+A(V)$.

The second construction is to take the interior $i_\ast(A)^\circ$ of $i_\ast(A)$.
\end{thm}

\begin{proof}
The statement is local, hence it suffices to show it on the stalks at
some point $p\in V$. Even more, we may pass to the completion, because
completion preserves equality of ideals, because completion of local rings is 
faithfully exact. We will show that both constructions
above are inverse to the restriction. This also shows that the second
construction is an inverse in the Zariski-local case, when there is not
necessarily a left inverse $\beta$.

Set $B=\beta^\ast(A)+A(V)$ and $C=i_\ast(A)^\circ$.
We can choose a system of regular local parameters such that
$V$ is the zero set of a subset of these parameters, and the images
of $\beta^\ast$ are constant on this subset of parameters. 
Then $B$ is a Villamayor algebra, and $B|_V=A$. Hence the first
construction is a right inverse operator for restriction.

Since $B$ is a Villamayor algebra contained in $i_\ast(A)$, we
also have $B\subseteq C$.
Let $T$ be a Villamayor algebra contained in $i_\ast(A)$.
We prove that $T_r\subseteq B_r$ for all $r\ge 0$, by induction on $r$.
For $r=0$, the statement is trivially true.
Let $r>0$. Let $a$ be an element of $T_r$. Then 
$a-\beta^\ast(i_\ast(a))$ is a sum of elements in 
$I(V)^j\Delta^j(T_r)$ for $j=1,\dots,r$ and $I(V)$ the ideal defining $V$ in $W$, by Taylor expansion
in the variables vanishing along $V$.
This sum is in $B_r$, because $B$ is a Villamayor algebra and
by induction hypothesis. But $\beta^\ast(i_\ast(a))$ is also
in $B_r$, hence $a\in B_r$. It follows that $T\subseteq B$
and consequently $B=C$.

Now let $T$ be a Villamayor algebra on $W$ such that $T|_V=A$.
Then $T\subseteq i_\ast(A)$, and it follows that $T\subseteq C$.
We prove that $\beta^\ast(A_r)\subseteq T_r$ for all $r\ge 0$, 
by induction on $r$.
For $r=0$, the statement is trivially true.
Let $r>0$. Let $b$ be an element of $A_r$. Because $T|_V=A$,
there is an element $a\in T_r$ such that $i_\ast(a)=b$. 
By Taylor expansion again, $a-\beta^\ast(b)$ is a sum of elements
in $I(V)^j\Delta^j(T_r)$ for $j=1,\dots,r$. This sum is
in $T_r$, and therefore $\beta^\ast(b)\in T_r$. 
It follows that $T=B=C$. This shows that the second construction
is a left inverse operator for restriction.
\end{proof}

%-------------------------------------------------

The theorem shows that the Villamayor algebra on $W$ does not carry more
information than its restriction to a subvariety.
A consequence is that different choices of subvarieties lead to isomorphic
restrictions. The following theorem is inspired by \cite{Wlodarczyk:05},
where a similar statement is shown for hypersurfaces.

%-------------------------------------------------

\begin{thm} \label{thm:wlo}
Let $W$ be a manifold. Let $A$ be a Villamayor algebra on $W$.
Let $V_j\subseteq W$, $j=1,2$, be two submanifolds of the same dimension.
Assume that there exist left inverses $\beta_j:W\to V_j$
for the inclusion maps $i_j:V_j\to W$, together with
isomorphisms $\phi:V_1\to V_2$ and $\psi:W\to W$ such that
$i_2\circ \phi=\psi\circ i_1$ and $\phi\circ\beta_1=\beta_2\circ\psi=\beta_2$.
Then $\phi^\ast(A|_{V_2})=A|_{V_1}$ and $\psi^\ast(A)=A$.
\end{thm}

\begin{proof}
We first observe that $\psi^\ast(A)$ is the algebra generated
by $\psi^\ast(\beta_2^\ast(A|_{V_2}))$ and $I(V_2)$ in degree~1. But
$\psi^\ast\circ\beta_2^\ast=\beta_2^\ast$, hence this is the algebra
generated by $\beta_2^\ast(A|_{V_2})$ and $I(V_2)$ in degree~1, and this
is exactly $A$. Then it follows also that
$\phi^\ast(A|_{V_2})=\phi^\ast({i_2}_\ast(A))={i_1}_\ast(\psi^\ast(A))=%
{i_1}_\ast(A)=A|_{V_1}$.
\end{proof}

%-------------------------------------------------

\begin{rem}
The maps $\beta_j$, $\phi$ and $\psi$ in Theorem~\ref{thm:wlo} need not exist in general,
even in the Zariski-local case. But they always
exist in the \'etale topology, i.e., we have such left inverses
and isomorphisms for the spectra of complete local rings.
Therefore, we will use the result only for proving
that certain algebras are equal after completion, but never for their construction.
\end{rem}

%-------------------------------------------------

\subsection{Orders.} 

The stability of Villamayor algebras under restriction allows to define order functions by their restriction to zooms.

\begin{defn} \label{def:ord}
Let $I$ be an ideal in $W$ with control $c$ and stalk $I_p$ at $p\in W$, let $A$ be a Rees algebra in $W$, and let $(A,d)$ be a gallimaufry in $(W,E)$. Define {\em order functions}  with values in $\Q\cup\{+\infty\}$ by 

\[\ord_{I,c}(p):={\frac{\ord_p(I_p)}{c}}, \]

\[\ord_A(p):=\min_{i>0}{\ord_{A_i,i}(p)},\] 

\[\ord_{A,d}(p)=\ord_B(p),\]

\noindent where $\ord_p$ denotes the order of ideals in the local ring $\OO_{W,p}$, and where $B$ is the restriction of $A$ to a zoom at $p$. The first two orders are defined on $W$, the third on the singular locus of $(A,d)$.
\end{defn}

%-------------------------------------------------

\begin{rem} \label{rem:ord} The next proposition guarantees that the
definitions make sense.
The order of an ideal is $\infty$ at $p$ if and only if 
the local ideal at $p$ is zero.
The order function of a gallimaufry is $\infty$ at $p$ if and only if
the algebra $A$ is locally equal to $A(V)$, for one or equivalently
for any zoom $V$ at $p$.
\end{rem}

%-------------------------------------------------

\begin{prop} \label{prop:ord}
Equivalent, respectively associated ideals with control and Rees algebras have the same order function.

The minimum in the definition of the order of a Rees
algebra is attained for a degree $i$ in some set of generating degrees.

If $B$ is the differential closure of $A$, then $\ord_A(p)=\ord_B(p)$
for all $p$ in the common singular locus of $A$ and $B$.

If $V_1,V_2$ are two zooms, and $B_1:=A|_{V_1}$ and $B_2:=A|_{V_2}$
are the two restrictions, then $\ord_{B_1}(p)=\ord_{B_2}(p)$ for all
$p$ in the singular locus of $(A,d)$.

The order function of a gallimaufry has values $\ge 1$. The denominator of any such value belongs to a set of generating degrees.
\end{prop}

\begin{proof}
The first two statements are straightforward. The third statement follows from
the fact that the set of all points $p$ with order at least some rational
number $q$ is equal to the intersection of the zero sets of the
ideals $\Delta^j(A_i)$ with $qi>j$. For the order function of $B$,
we have to take all ideals of the form $\Delta^j(B_i)$ with $q i>j$,
or, equivalently, all ideals of the form $\Delta^{j+k}(A_{i+k})$
with $i>q j$ and $k\ge 0$, because $B_i$ is the sum of
all $\Delta^k(A_{i+k})$, $k>0$. But if $q\ge 1$, the inequality
$q i>j$ implies $q (i+k)>(j+k)$, hence the two intersections are equal.

The fourth statement is an immediate consequence of Theorem~\ref{thm:wlo}.

The last statement is a consequence of the second and third statement.
\end{proof}

%-------------------------------------------------

\begin{rem}
By Lemma~\ref{lem:later}, any set of generating degrees for $A$ as a Villamayor algebra
is also a set of generating degrees for $A'$ as a Villamayor algebra (but in general
not as a Rees algebra). As a consequence, the maximum of the order
function cannot drop infinitely often.
\end{rem}

%-------------------------------------------------

\begin{defn}
A Rees algebra or a gallimaufry is called {\em tight} if its order
function is $\le 1$; in the gallimaufry case, this implies that it is equal to $1$.
\end{defn}

\begin{defn}
If $(A,d)$ is a gallimaufry in a habitat $(W,E)$ so that $(A,d-1)$ is again a gallimaufry in $(W,E)$, then $(A,d-1)$ is called the {\em descent} of $(A,d)$.
\end{defn}

%-------------------------------------------------

The next lemma is an immediate consequence of Proposition~\ref{prop:trf}.

\begin{lem} \label{lem:comm}
If $(A,d-1)$ is the descent of $(A,d)$, and $\widetilde A'$ is the differential closure
of the transform $A'$ of $A$ under a blowup of $W$ along a transversal center in the
singular locus, then $(\widetilde A',d-1)$ is the descent of $(\widetilde A',d)$.
\end{lem}

%\begin{proof} Straightforward. \end{proof}

%-------------------------------------------------

\begin{lem} \label{lem:tight}
If a gallimaufry has a descent, then it is tight. Conversely, if $E$ is empty, then every tight gallimaufry with $d>0$ has a descent. The transform of a tight gallimaufry under blowup is tight. 
\end{lem}

\begin{proof}
Assume that $(A,d)$ has a descent. Let $p$ be a point in the singular locus.
Then there exists a zoom $V_1$ of dimension $d-1$ at $p$. It is easy to construct
a zoom $V$ of dimension $d$: just leave away one of the defining equations
in a system of regular parameters that defines $V_1$ and such that 
every hypersurface $E_i$ in $E$ containing $p$ is defined by one of the 
parameters -- such a system exists by transversality of $V_1$. 
Let $a\in\OO_{W,p}$ be the parameter which has been left out to define $V$.
Then $a$ considered as an element in $\OO_{V,p}$ is an element of order~1 
in the restriction $A|_{V}$. Hence $\ord_{A,d}(p)=1$ and
it follows that $(A,d)$ is tight.

Now, assume that $E$ is empty, and that $(A,d)$ is tight with $d\ge 1$.
Let $p$ be a point in the singular locus. Let $V$ be a zoom at $p$,
closed in some open neighborhood $U$ of $p$. Since $A|_V$ has order $1$
at $p$ (note that this can only happen for $d\ge 1$), 
there exists an element $a\in A_1$ such that $a$ restricted to $V$
has order $1$. The zero set $Y$ of $a$ in $V$ is then non-singular
at $p$. After shrinking $U$ we may assume that $Y\cap U$
is non-singular -- and therefore transversal, since $E$ is empty. Obviously
$A(Y)$ is contained in $A$ restricted to $U$, hence $Y$ is a zoom
at $p$ of dimension $d-1$. It follows that $(A,d-1)$ is a descent.

In order to prove the last statement, assume again that $(A,d)$ 
is tight. This condition is not related to $E$, so we may assume
that $E$ is empty. Therefore, $(A,d)$ has a descent $(A,d-1)$.
If $\widetilde A'$ is the differential closure of the transform $A'$ of $A$, then
$(\widetilde A',d-1)$ is the descent of $(\widetilde A',d)$, by Lemma~\ref{lem:comm}.
But then $(\widetilde A',d)$ is tight, as we have shown above.
\end{proof}

%-------------------------------------------------

\subsection{\Factors.} 

Descent only works for tight gallimaufries. It is possible to produce
tight gallimaufries from non-tight ones.
The easiest way to see this
is in terms of ideals with control: when the control $c$ is replaced by
the maximum order of the ideal, then the ideal with the new control is tight.
Changing the control also changes the transform of the ideal under blowup:
the new ideal differs from the old one by a \factor{} supported on the exceptional divisor.  This leads us to the problem of defining \factors{} for gallimaufries. 

%-------------------------------------------------

\begin{defn} \label{def:ft}
Let $(W,E)$ be a habitat. A {\em monomial} on $(W,E)$ is a formal sum $S$ of the
hypersurfaces in $E$ with non-negative rational coefficients.
A monomial induces a {\em monomial function} $s:W\to\Q$ taking $p$ 
to the sum of the coefficients of the hypersurfaces through $p$. 

Let $(I,c)$ be an ideal with control with order function $\ord:W\to\Q$,
and let $S$ be a monomial with monomial function $s:W\to\Q$. 
We say that $S$ is a {\em \factor{}} of $(I,c)$ if and only if 
$\ord(p)\ge s(p)$ for all $p\in W$. In the same way, we define \factors{}
of Rees algebras.
\end{defn}

%-------------------------------------------------

\begin{lem} \label{lem:quot}
Let $(I,c)$ be an ideal with control.
Let $S$ be a \factor{} of $(I,c)$.
Then there exists an ideal with control $(J,b)$, an integer $k>0$,
and a monomial ideal $N$, such that the order function of
$N$ is $k\cdot s$ and $I^{bk}=N^{bc}J^{ck}$. The ideal $(J,b)$ is unique
up to equivalence. Its order function is $\ord_{I,c}-s$.
\end{lem}

\begin{defn}
The ideal $(J,b)$ is called the {\em quotient} of $(I,c)$
by $S$. The quotient of a Rees algebra by a \factor{} is defined by passing
to the associated ideal with control, taking the quotient, and returning to
the associated Rees algebra.\footnote{It seems mandatory here to pass from algebras to associated ideals in order to define \factors.}
\end{defn}

\begin{proof}
Let $k$ be a common denominator of all coefficients appearing in $S$.
Let $N$ be the monomial ideal defined by the integer-valued monomial $kS$.
For $E_i\in E$,  let $\frac{r}{k}$ be the coefficient of $E_i$ in $S$. 
Then $\frac{r}{k}\le\frac{\ord_p(I_p)}{c}$ for all $p\in E_i$.
It follows that $\ord_p(I_p^{k})\ge rc$ for all $p\in E_i$.
Then the ideal $I({E_i}^{rc})$ of ${E_i}^{rc}$ is a \factor{} of every section in $I^k$. 
The product of all these \factors{} is $N^{c}$, and by dividing them out
we get $J$. Finally we set $b=ck$; then the desired equations for
ideals and order functions are fulfilled for $(J,b)$.
\end{proof}

%-------------------------------------------------

\begin{lem} \label{lem:stab}
Let $(I,c)$ be an ideal with control. Let $S$ be a \factor{} of $(I,c)$
and let $(J,b)$ be the quotient. Let $Z$ be a transversal subvariety
contained in the singular locus of both $(I,c)$ and $(J,b)$.
Let $I'$ and $J'$ be the controlled transforms on $W'$ under the blowup of $W$ along $Z$.
Then there is a \factor{} $S'$ of $(I',c)$ such that $(J',b)$ is
equivalent to the quotient. 

Similarily, the quotient of a Rees algebra $A$ by some \factor{} transforms under blowup
to a quotient of the transform of $A$.\footnote{Compare this with the transformation formula under blowup for the combinatorial handicap of a mobile in \cite{Encinas_Hauser:02}.}
\end{lem}

\begin{proof}
By pulling back the two sides of the equation $I^{bk}=N^{bc}J^{ck}$,
we obtain $(I')^{bk}M^{cbk}=f^\ast(N)^{bc}(J')^{ck}M^{bck}$, where $M$
is the ideal sheaf of the exceptional divisor $D$. Since $M$ is invertible,
we may cancel it on both sides. Then the monomial of $f^\ast(N)$ divided by $k$ is
the desired \factor{} of $(I',c)$.
\end{proof}

%-------------------------------------------------

\begin{defn}
Let $(A,d)$ be a gallimaufry. Let $S$ be a monomial.
We say that $S$ is a \factor{} of $(A,d)$ if and only if for any $p$
in the singular locus, there exists a zoom $V$ such that $S$ considered
as a monomial in $V$ is a \factor{} of $A|_V$. 
\end{defn}

%-------------------------------------------------

\begin{rem}
If $V_1$ and $V_2$ are two zooms at $p$, then there exist analytic
left inverses $\beta_j:W\to V_j$
for the inclusion maps $i_j:V_j\to W$,
and analytic isomorphisms
$\phi:V_1\to V_2$ and $\psi:W\to W$ such that
$i_2=\psi\circ i_1$ and $\phi\circ\beta_1=\beta_2\circ\psi=\beta_2$
and such that $\psi$ fixes the hypersurfaces in $E$, as a consequence
of transversality. By Theorem~\ref{thm:wlo}, it follows that
the property for some monomial being a \factor{} of a gallimaufry 
does not depend on the choice of the zoom.
\end{rem}

\begin{rem}
It is not possible to recognize \factors{} of a gallimaufry purely by looking
at the order function. The reason is that the domain of the order function is
only the singular locus (which may be empty), and this is too small
for this purpose.
\end{rem}

\begin{defn}
A \factor{} $S$ of $(A,d)$ with monomial function $s$ is called {\em exhaustive} if and only if $s(p)=\ord(p)$ for all $p$ in the singular locus.\footnote{This notion parallels the concept of complete \factors{} from the section {\em Salmagundy}.}
\end{defn}

%-------------------------------------------------

\begin{lem} \label{lem:ex}
Let $S$ be an exhaustive \factor{} of the gallimaufry $(A,d)$, with
monomial function $s$ and monomial ideal $(N,k)$ (i.e., $k$
is a common denominator of all values and $N$ is the monomial ideal
with order function $ks$). Let $(A',d)$ be the transformed gallimaufry
under some blowup along a transversal center inside the singular locus.
Let $S'$ be (the ideal of) the transform of $(N,d)$. Then $S'$ is an exhaustive \factor{} of $(A',d)$.
\end{lem}

\begin{proof}
Let $p$ be a singular point, and let $V$ be a zoom at $p$.
Let $(I,c)$ be the ideal associated to $A|_V$ for some zoom $V$.
Then $(I,c)$ is equivalent to $(N|_V,k)$ locally at $p$.
It follows that the transforms $(I',c)$ and $(N'|_{V'})$ are also
equivalent (where $V'$ is the strict transform of $V$). Therefore the
order functions of $(I',c)$ and of $(N',k)$ are equal.
\end{proof}

An analogue of Lemma~\ref{lem:quot} for gallimaufries
would certainly be useful for dividing out \factors.
Extra care is necessary to make the construction independent
of the choice of the zoom.

%-------------------------------------------------

\begin{thm} \label{thm:q}
Let $M=(W,E)$ be a habitat.
Let $(A,d)$ be a gallimaufry on $M$. Let $S$ be a \factor{} of $(A,d)$
with monomial function $s$.
Let $q>0$ be a positive number. 
Then there exists a gallimaufry $(B,d)$ such that
\[ \sing(B,d)=\{ p\in\sing(A,d)\mid \ord_{A,d}(p)-s(p)\ge q\} \]
and
\[ \ord_{B,d}(p)=\min\left(\ord_{A,d}(p),\frac{\ord_{A,d}(p)-s(p)}{q}\right) \]
for all $p$ in $\sing(B,d)$.
\end{thm}

We call $(B,d)$ the {\em quotient} of $(A,d)$ by $S$ scaled by $q$.

\begin{proof}
We construct the sheaf $B$ locally, so we assume that we have a zoom $V$.
Then $S$ (considered as a monomial in $V$) is a \factor{} of $A|_V$.
Let $(I,c)$ be the ideal associated to $A|_V$.
By Lemma~\ref{lem:quot}, there is a quotient ideal $(J,c)$.
Let $C$ be the differential closure of the Rees algebra associated
to $(J^n,cm)$, where $\frac{m}{n}=q$. Then we define $B$ as the extension
of $A|_V+C$. The claimed equalities of the singular loci and order functions
are then fulfilled. 

We need to show that the result does not depend on the choice of
the zoom $V$ (otherwise we do not get a sheaf). 
Let $V_1,V_2$ be two zooms at $p$, where $p$ is a point in the
singular locus. We denote the two results of the above construction
by $B_1,B_2$. Let us assume that there exist analytic left inverses
$\beta_1:W\to V_1$ and $\beta_2:W\to V_2$
for the inclusion maps $i_1:V_1\to W$ and $i_2:V_2\to W$,
as well as isomorphisms $\phi:V_1\to V_2$ and
$\psi:W\to W$ such that
$i_2=\psi\circ i_1$ and $\phi\circ\beta_1=\beta_2\circ\psi=\beta_2$
and $\psi$ maps any divisor in $E$ into itself.
By Theorem~\ref{thm:wlo}, the isomorphism $\phi^\ast$ takes
$A|_{V_2}$ to $A|_{V_1}$. Let $C_i=B_i|_{V_i}$ for $i=1,2$.
Then $C_i$ may be obtained as above (passing to ideals with control,
taking quotients changing the control, passing back to algebras, 
adding $A|{V_i}$), and it follows that $\phi^\ast$ takes $C_2$ to $C_1$.

By the first construction of the extension operator in Theorem~\ref{thm:aus},
$B_1=\beta_1^\ast(C_1)+A(V_1)$ and $B_2=\beta_2^\ast(C_2)+A(V_2)$.
By $A|_{V_1}\subseteq C_1$ and monotonicity of extension, we have $B_1\supseteq A\supseteq A(V_2)$
and similarily $B_2\supseteq A\supseteq A(V_1)$. For the other summands, we have
\[ \beta_2^\ast(C_2)=\beta_1^\ast\circ\phi^\ast(C_2)= \beta_1^\ast(C_1) , \]
and it follows that $B_1=B_2$.

In general, there do not always exist left inverses $\beta_1,\beta_2$
and isomorphisms $\phi,\psi$ as above, not even in the local rings. 
But they do exist after completion. As completion is
faithfully flat, we get $B_1=B_2$ also in this case.
\end{proof}

%-------------------------------------------------

\begin{lem} \label{lem:q}
Let $(A,d)$ be a gallimaufry. Let $S$ be a \factor{} of $(A,d)$.
Let $q>0$ be a rational number.
Let $(B,d)$ be the quotient of $(A,d)$ by $S$ scaled by $q$. 
Let $Z$ be a transversal subvariety
contained in the singular locus of $B$,
and let $(A',d)$ and $(B',d)$ be the transformed gallimaufries.
Then there is a \factor{} $S'$ of $(A',d)$ such that $(B',d)$ is
the quotient of $(A',d)$ by $S'$ scaled by $q$.
\end{lem}

\begin{proof}
This is a straightforward consequence of Lemma~\ref{lem:stab}
and Lemma~\ref{lem:rescom}.
\end{proof}

%-------------------------------------------------

\subsection{Axioms.} 

We conclude this section by resuming the main properties of
gallimaufries and operations on them. This review may also be seen
as an ``axiomatic characterization'' of gallimaufries. It might well be
that one can replace the algebraic realization of gallimaufries presented
here by a different one, but as long as the axiomatic characterization
is fulfilled, the construction can be used for translating it
into the combinatorial game {\em Salmagundy}.

%-------------------------------------------------

\begin{thm} \label{thm:axioms}
Any gallimaufry of dimension $d$ on a habitat $(W,E)$ defines a closed subset, the singular locus, and an upper semicontinuous order function from the singular locus
to $\Q\cup\{\infty\}$. Its values are at least 1. If all values are
equal~1, then the gallimaufry is tight.

If $Z$ is a transversal subvariety inside the singular locus, then
the transform of the gallimaufry under the blowup of $W$ along $Z$ exists on the transformed habitat. The transform of a tight gallimaufry is tight.

If the gallimaufry is tight and $E$ is empty,
it admits a descent. The descent gallimaufry has the
same singular locus, but possibly a different order function; its
dimension is $d-1$. The transform under blowup does not change the
dimension $d$. The transform of the descent is the descent of the transform.

The subset of the singular locus with order $\infty$ is non-singular
of pure dimension $d$. If a $d$-dimensional component is transversal to $E$,
then the transform under the blowup along this component has no singular points in the exceptional
divisor. 

The existence of a \factor{} $S$ of a gallimaufry implies that
the singular locus contains the set of all points where the monomial function
is at least 1, and that the order function is greater than or equal to
the order function. If equality holds everywhere, i.e., if the \factor{} is
exhaustive, then the transform of the \factor{} is an exhaustive \factor{}
of the transform of the gallimaufry. The trivial monomial~1 is always
a \factor.

If a gallimaufry $G$ with order function $\ord_G$ has a \factor{} $S$,
and $q>0$ is a rational number, then there exists the quotient
gallimaufry scaled by $q$. Its singular locus is the set of all points $p$ in the singular locus of $G$ such that $\ord_G(p)-s(p)\ge q$,
where $s$ is the monomial function defined by $S$. Its order function
is the minimum of $(\ord_G-s)/q$ and $\ord_G$. The transform
of the quotient is a quotient of the transform by some monomial, scaled
by $q$.
\end{thm}

\begin{proof}
Almost everything has been already proven, and the remaining assertions are
easy consequences:
the singular locus of the gallimaufry $(A,d)$ is the singular locus of $A$
as defined in Definition~\ref{def:rees}. It is closed because it is the
zero set of $A_1$ (see the proof of Lemma~\ref{lem:later}). The order
function is defined in Definition~\ref{def:ord} and Proposition~\ref{prop:ord}.
The existence of the transform is Proposition~\ref{prop:trf}. 
The stability of tightness under blowup is stated in Lemma~\ref{lem:tight}.
The existence of the descent in case $E$ is empty is also stated
in Lemma~\ref{lem:tight}. The equality of the singular loci is obvious,
because the singular locus depends only on $A$. Commutativity of
descent and transform is stated in Lemma~\ref{lem:comm}. 
The structure of the set of points with order $\infty$ follows from the
remark before Proposition~\ref{prop:ord}. The non-existence of singular points in the
exceptional divisor under blowup of a component follows
from the same remark together with the remark preceding Theorem~\ref{thm:aus}. The implications of the existence of a \factor{} on the singular locus, on the order function and  the fact
that $1$ is always a \factor{} are easy consequences of the definition.
Stability of exhaustive \factors{} is stated in Lemma~\ref{lem:ex}.
The existence and properties of the quotient is Theorem~\ref{thm:q}, 
and commutativity of quotient and transform is Lemma~\ref{lem:q}.
\end{proof}

%-------------------------------------------------
%             TRANSCRIPTION
%-------------------------------------------------

\section{Transcription}

We show in this section how the resolution problem represented by a gallimaufry can be transcribed into a scenario of our resolution game Salmagundy so that the transforms of the gallimaufry under blowups, descents and quotients correspond to the respective transforms of the scenario. This permits to transfer the search for a resolution algorithm for a given gallimaufry to the search of a winning strategy for Dido for the associated scenario.
By applying the reverse transcription from scenarios to gallimaufries, Dido resolves unconsciously a gallimaufry, despite the fact that she has never heard of algebraic varieties, ideal sheaves and blowups (after all, sheaf theory was not yet invented at the times of
Carthago).

The transcription has to be done in both directions. Dido's moves in the game will be reread geometrically: In particular, the choice of the nodes which represent the centers of her blowup move will be interpreted as the choice of a transversal subvariety of the habitat of a gallimaufry, i.e., the center of the blowup of the habitat, which is contained in the singular locus and induces the transform of the gallimaufry. 

If Dido opens a quest, Mephisto has to respond with a scenario. In this, he just has to obey the rules of the game. But for the transcription between the resolution of gallimaufries and the winning strategy for the game, Mephisto's answer will be given as the scenario defined by the respective modification of the gallimaufry which is specified by the quest (e.g., for a descent quest, it will be the scenario corresponding to the descent gallimaufry). It then has to be checked that this answer also obeys the rules of the game. In this sense, Mephisto's will assume the role of a ``Knecht'' of the algebraic geometer who tries to resolve a singularity: His response in the game is dictated by the reality of the singularities.\footnote{Of course, when taking the game abstractly, Mephisto may choose the moves at his taste.} 

As resolved scenarios correspond to conquered quests, it follows that a winning
strategy for the game induces a resolution strategy for gallimaufries.\\

A {\em stratification} of a topological space $X$ is
a finite partition of $X$ into locally closed, irreducible subsets, 
called {\em strata}, such that
the closure of any stratum is a union of strata. 
A stratum $A$ is {\em adjacent} to a stratum $B$ if it is contained
in the closure of $B$. A stratification $T'$ is a refinement of $T$
if every stratum of $T$ is a union of strata of $T'$.

For any finite
collection of closed sets, there is coarsest stratification such that
any closed set in the collection is a union of strata. Its strata
are the irreducible components of all intersections of the closed sets
in the collections and their complements. We call it the stratification
induced by the collection of closed sets.
Similarily, any upper semicontinuous function from $X$ to a finite, partially 
ordered set induces a stratification.

\begin{defn}
Let $(W,E)$ be a habitat, and let $(A,d)$ be a gallimaufry on $W$.
Let $T$ be a stratification of $W$ such that any divisor in $E$
is a union of strata and such that the order function of $(A,d)$ is constant 
in each stratum.
Then the {\em induced scenario} $C=C(T,A,d)$ is defined as follows:
\begin{enumerate}
\item the nodes of the underlying graph $\Gamma$ are the strata of $T$,
	and the directed edges between nodes are given by the 
	adjacency relation;
\item the dimension of a node is the dimension of the stratum;
\item the singular set is the set of strata with order $\ge 1$;
\item the order of a node is is the order of the stratum;
\item the jibs are the dense strata of the divisors $E_i$ in $E$;
\item the transversal nodes are the strata whose closure is transversal 
        to $E$;
\item the bound $B$ is an least common multiple for the degrees of 
	elements in a finite generating set of the Villamayor algebra $A$;
\item the dimension is $d$;
\item the set ${\cal M}$ is the set of all monomial factors of $A$.
\end{enumerate}
\end{defn}

%-------------------------------------------------

\begin{lem} \label{lem:blow}
Let $Z$ be a subvariety of $W$ which is transversal to $E$ and contained in the
singular locus of $(A,d)$. Assume that $Z$ is a union of strata in $T$. Let $(A,d)$ 
be the transformed gallimaufry under the blowup $f:W'\to W$ of $W$ along $Z$, 
and let $T'$ be a stratification of $W'$ at least as fine as the one 
induced by the set of exceptional divisors, the singular locus, 
the order function, and the preimages of strata in $T$. Then the scenario 
$C(T',A',d)$ is a transform of $C(T,A,d)$ under the blowup with center 
the node $\cc$ corresponding to the dense stratum of $Z$.
\end{lem}

\begin{proof}
We have to check that Rule~\ref{rul:trans} and Rule~\ref{rul:refblow}
are fulfilled. Most issues are straightforward, so we treat
here only the more interesting ones.
Rule~\ref{rul:trans}, Issue~\ref{srul:indim} says that
$\dim(f^{-1}(Y))=\dim(Y)+\mathrm{codim}(Z)-1$,
if $Y$ is a locally closed subvariety of the center $Z$. 
This is a well-known property
of blowups along non-singular subvarieties.
Rule~\ref{rul:refblow}, Issue~\ref{srul:adm} says that the strict 
transform of a transversal subvariety or the total transform of a 
subvariety contained in the
center is transversal to the exceptional divisor. 
This is a well-known property of blowups
along centers transversal to an already existing exceptional divisor.
Item~\ref{srul:tight} is an immediate consequence of Theorem~\ref{thm:axioms},
stating that the transform of a tight gallimaufry is tight.
\end{proof}
%-------------------------------------------------

\begin{lem} \label{lem:rel}
Let $C$ be the scenario induced by a gallimaufry $(A,d)$ on a habitat
$(W,E)$. Let $C'$ be the scenario induced by the gallimaufry $(A,d)$
considered on the habitat $(W,E')$, where $E'=E\setminus \{H\}$ for some $H\in E$.
Then $C'$ is a relaxation scenario for $C$, releasing the node corresponding to
$H$.
\end{lem}

\begin{proof}
By comparing with Rule~\ref{rul:relax}, it is clear that it suffices
to check the following property: if $Z$ is transversal to $E\setminus\{H\}$,
and $Z$ is either contained in $H$ or disjoint to it,
then $Z$ is transversal to $E$. But this is obvious. 
\end{proof}

%-------------------------------------------------

\begin{lem} \label{lem:des}
Let $C$ be the scenario induced by a tight gallimaufry $(A,d)$ on a habitat
$(W,E)$. Let $C'$ be the scenario induced by the descent gallimaufry 
$(A,d-1)$. Let $T'$ be refinement of the stratification $T$ of $C$
such that the order of $(A,d-1)$ is constant along each stratum. Then
$C(T',A,d-1)$ is a descent scenario for $C(T,A,d)$.
\end{lem}

\begin{proof}
Obvious.
\end{proof}

%-------------------------------------------------

\begin{lem} \label{lem:qg}
Let $(A,d)$ be a gallimaufry on a habitat $(W,E)$, and let $T$ be a
stratification of $W$ such that all hypersurfaces of $E$ and the singular
locus are unions of strata, and such that the order function is constant in each stratum.
Let $S$ be a \factor{} of $(A,d)$. Let $q\in\Q_{\ge 0}$,
and let $(A',d)$ be the quotient gallimaufry. Then $C(T,A',d)$ is 
a $q$-quotient scenario for $C(T,A,d)$), and $S$ is the
divided \factor.
\end{lem}

\begin{proof}
This is a direct consequence of Theorem~\ref{thm:axioms}.
\end{proof}

%-------------------------------------------------

\begin{thm}\label{thm:resgall}
If there is a winning strategy for Dido in the game Salmagundy, gallimaufries
admit a resolution.
\end{thm}

\begin{proof}
Let $(A,d)$ be a gallimaufry on an habitat $(W,E)$. 
Let $G$ be the set containing as a single element this gallimaufry.
Throughout, it will be a set of gallimaufries on the same manifold $W$
(the set $E$ may vary).
Let $T$ be a stratification of $W$ induced by the divisors in $E$, 
by the singular loci of the gallimaufries in $B$,
and by the order functions of the gallimaufries in $B$.
Let $C=C(T,A,d)$ be the induced scenario, and let $B$ be the underlying 
board. This is the initial scenario that the Umpire provides for the game. 
By assumption, Dido has a winning strategy, and she uses it to win this game.

In her move, Dido may either specify a transversal and singular
node $\cc$ for a blowup move, apply a one way quest, or place a call for a quest. In case of blowup, the closure of the stratum corresponding to $\cc$ is taken as the center $Z$
of a blowup $W'\to W$ of $W$. Let $G'$ be the set of all
transformed gallimaufries $(A',d)$.
Let $T'$ be a stratification induced by the divisors in $E'$, the singular
loci, the order functions, and the preimages of strata in $T$.
Then, for any gallimaufry $(A',d)$ in $G'$, let $C'=C(A',d,T')$ 
be the induced scenario. Then $C'$ is a transform
of $C$ by Lemma~\ref{lem:blow}, and Mephisto's move is given by the
blowup transforms of the board and the scenarios $C'$.

If the move of Dido was a relaxing call releasing the jib $h$,
the respective hypersurface $H$ is removed from $E$ providing a new habitat
$(W',E')$ and a new gallimaufry $(A',d)$, which is just $(A,d)$ considered on this
new habitat. It is added to $G$.
The stratification $T$ is also compatible with $(A',d)$, and by
Lemma~\ref{lem:rel}, the scenario $C(T,A',d)$ is a relaxation response
for $C(T,A,d)$. Mephisto's move consists in giving the trivial refinement,
the unaltered scenarios for all quests in $G$, the scenario
$C'$ for the relaxation quest, and the scenarios $C$ for the main quest and $C'$ for the relaxation quest.

If the move was a call for a descent quest, then $(A,d)$ must be tight and $(A,d-1)$
is a descent gallimaufry. It is added to $G$. 
Let $T'$ be the refinement induced by the new
order function. By Lemma~\ref{lem:des}, $C(T',A,d-1)$ is a descent scenario
for $C$. Mephisto's move consists in refining the board and the scenarios
for the old quests and providing a response $C(T',A,d-1)$ for the descent quest.

If the move was the construction of a transversality quest, then the respective hypersurfaces are removed from $E$ providing a new habitat $(W,E')$ and a new gallimaufry $(A',d)$, where $A'$ is the sum of $A$ and the Villamayor algebra generated (in degree~1) by the sum of the ideals of these hypersurfaces. The stratification $T$ is
compatible with $(A',d)$, and the scenario $C(T,A',d)$ is a transversality scenario.

If the move was a \factorization{} call with divided \factor{} $m$, 
then $m$ is also a \factor{} of the gallimaufry $(A,d)$. Let $(A',d)$
be the quotient gallimaufry, appropriately scaled. 
By Lemma~\ref{lem:qg}, $C(T,A',d)$ is a scaled quotient scenario.

Now it is again Dido's turn. Since Dido has a winning strategy, she wins
the game after a finite number of steps. By then, the singular locus 
of the successive transforms of $(A,d)$ must have become empty, and the
gallimaufry is resolved.
\end{proof}

%-------------------------------------------------
%           WINNING STRATEGY
%-------------------------------------------------

\section{A Winning Strategy}

In this section we prove that Dido has a winning strategy for the 
resolution game Salmagundy defined in section~\ref{sec:rules}. Again, 
there are no references to algebraic concepts such as ideals, varieties etc.

%-------------------------------------------------

We say that a quest is {\em strictly} won if the blowup centers chosen by Dido always lie 
in the singular set ${\cal S}$.\footnote{So being remote from ${\cal S}$ is excluded. To have the centers inside ${\cal S}$ is required for the induction argument to work, see the proof of Lemma~\ref{lem:descent}.} Our goal is to show that any quest admits a strict winning strategy. Along the way, it will be necessary to introduce and play various auxiliary quests, 
whose winning strategies require centers which may lie outside the singular set of some of the other auxiliary quests. It will, however, be ensured that the centers always lie in the singular set of the main quest.

%-------------------------------------------------

\begin{lem} [Hironaka] \label{lem:mon}
There is a strict winning strategy for monomial scenarios.\footnote{The lemma asserts that a quest can be strictly won if its initial scenario is monomial, i.e., has a complete \factor. This result corresponds to the so called resolution of varieties in the monomial case.}
\end{lem}

%-------------------------------------------------

\begin{proof}
Let $\mathfrak{Q}$ be a quest with a complete \factor{} $m:{\cal H}\to\Q$.
Let $C$ be a scenario for $\mathfrak{Q}$ on some board $\Gamma$,
and assume that $\mathfrak{Q}$ is not won yet (i.e., the singular set $\cal S$ of $C$ 
is not empty). We define a critical set for $C$ as a subset ${\cal K}$ of ${\cal H}$ such that 
there is a singular node of $C$ smaller than or equal to every node in 
${\cal K}$.\footnote{Singular nodes are expected candidates for centers of blowups. 
Dido has to ensure by auxiliary blowups that the center of preference becomes 
transversal to all hypersurfaces of $\cal H$ it meets, i.e., those from nodes in ${\cal K}$. 
Only after this preparation it can be chosen as a center. 
The first step in the resolution is therefore to separate the stratum Dido 
would like to choose as center from the hypersurfaces from ${\cal K}$. 
In the geometric situation, this is done by means of the transversality 
ideal \cite{Encinas_Hauser:02}, \cite{Encinas_Villamayor:97b}.} 
For monomial scenarios as here (i.e., those with a complete \factor{} $m$), 
this is the case if and only if $\sum_{h\in {\cal K}}m(h)\ge 1$. Observe that as long as 
$\cal S$ is non-empty there exist critical sets. Let $N=N_{\cal K}$ be the set of singular nodes of $C$ that are maximal among the singular nodes lying below all $h\in {\cal K}$. 
By Rule~\ref{rul:scen}, Issue~\ref{srul:new}, the nodes in $N$ have dimension $d-\mathrm{card}({\cal K})$, and any two nodes in $N$ have no node in $N$ below both of them. By Rule~\ref{rul:refblow}, Issue~\ref{srul:istr}, it follows that the nodes in $N$ are transversal, i.e., belong to $\cal T$. By Rule~\ref{rul:refblow},
Issue~\ref{srul:btr}, the iterated blowup transform of $C$ with centers all nodes in $N$ (in any order; the individual blowups do not effect each other because
there is no common node below two nodes in $N$) achieves that
the transforms of the jibs in ${\cal K}$ do not have a common node below them.

An {\em elementary step} for a choosen minimal critical set ${\cal K}$ consists in blowing up successively all nodes in $N=N_{\cal K}$ (this, of course, represents a sequence of moves of Dido, with respective responses by Mephisto). We claim that quests with a complete \factor{} can be won by a concatanation of suitable elementary steps.

Define the {\em multiplicity} of a critical set ${\cal K}$  as $\sum_{h\in {\cal K}}m(h)$.
As long as $\mathfrak{Q}$ is not yet resolved, Dido can and will choose a minimal critical set ${\cal K}$ for the actual scenario $C$ of $\mathfrak{Q}$. She then applies the corresponding elementary step to $C$. Let us assume that the elementary step consists only of one blowup; the general case is similar. Let $C'$ be the blowup transform of $C$. Then any critical set of $C'$ is either the set of transforms of the jibs of some critical set ${\cal K}_1$ of $C$ which is not a superset of ${\cal K}$ -- we call this an ``old critical set'' --, or the set of transforms of the jibs of some critical set ${\cal K}_2$ of $C$ which is a superset of ${\cal K}$, where one of the jibs in ${\cal K}$ is replaced by the exceptional jib $e=i(\cc)$ -- we call this a ``new critical set''. The multiplicity of an old critical set of $C'$ is equal to the multiplicity of the preceding critical set of $C$; the multiplicity of a new critical set of $C'$ is smaller than the multiplicity of the preceding critical set of $C$, because 

\[m(e)=\sum_{h'\in {\cal K}} m(h')-1 =m(h)+\sum_{h'\in {\cal K}\setminus\{ h\}} m(h')-1 <m(h),\]

\noindent where the last inequality is a consequence of the minimality of ${\cal K}$.
If we identify the old critical sets of $C'$ with the corresponding critical sets on $C$, then we can say that an elementary step replaces some critical sets by new critical sets, each of them of smaller multiplicity than the one which was replaced. It follows that in any sequence of elementary steps all critical sets disappear after finitely many iterations. This implies that the singular set $\cal S$ of $C$ has become empty, i.e., that $C$ is resolved and $\mathfrak{Q}$ is won.
\end{proof}

%-------------------------------------------------

\begin{lem} [Relaxation] \label{lem:rtight} 
Let $n\ge 0$ be an integer. If there is a strict winning strategy for tight quests of dimension $n$, then there is a strict winning strategy for all quests of dimension $n$.
\end{lem}

%-------------------------------------------------

\begin{proof}
Let $\mathfrak{Q}$ be a quest of dimension $n$. In order to win $\mathfrak{Q}$, Dido keeps track of a \factor{} $m:{\cal H}\to\Q$ of the scenarios of $\mathfrak{Q}$, for which the bound $B$ of $C$ is a common denominator of all values of $m$. Initially, $m$ may be chosen as zero; this is always a \factor. Assume that we are at a certain stage of the game, with actual scenario $C$ and \factor{} $m$ of  $\mathfrak{Q}$.

We distinguish two cases: First, $m$ is a complete \factor.
Then $\mathfrak{Q}$ can be won by Lemma~\ref{lem:mon}.

Second, $m$ is not a complete \factor. Then $\ord(s)>m(s)$ for at
least one singular node $s$ of $C$. Let $q>0$ be the maximal value 
of $\ord(s)-m(s)$ for $s\in{\cal S}$. We will show that Dido has a strategy to make this maximum drop.

If $q=\infty$, the maximal singular nodes for which $\ord(s)=\infty$ have dimension $d$ and are transversal, by Rule~\ref{rul:scen}, Issues~\ref{srul:max} and \ref{srul:istr}. By Rule~\ref{rul:refblow}, Issue~\ref{srul:btr}, the blowup of
one of these nodes reduces by $1$ the number of $d$-dimensional nodes
with order $\infty$. By a finite number of steps, Dido reaches a scenario $C$ of $\mathfrak{Q}$ where $q$ is finite. Along the way, she lifts the \factor{} $m$ according to  Rule~\ref{rul:comm}, Issue~\ref{srul:qtr}.

Dido calls now a \factorization{} quest creating thus the $q$-quotient $\mathfrak{Q}_1$ of $\mathfrak{Q}$ with respect to the \factor{} $m$. Let $C_1$ be the obtained scenario. It is tight, and its singular set consists of all nodes $s$ with $\ord(s)-m(s)=q$. After any blowup along a center which is in the singular set ${\cal S}_1$ of $\mathfrak{Q}_1$, the transform $(C_1)'$ of $C_1$ is tight, and it is also a $q$-quotient scenario for $C'$, the transform of $C$, by some \factor{} $m'$, see Rule~\ref{rul:comm}, Issue~\ref{srul:qtr}. Note that the bound $B'=B$ is again a common denominator for the values of $m'$. It follows that $\ord(s')-m'(s')\le q$ for all $s'$ in the singular set of $C'$, and equality holds for the singular nodes of $(C_1)'$.

By assumption, $\mathfrak{Q}_1$ can be won by a strict winning strategy, so Dido uses it
to win $\mathfrak{Q}_1$. After these blowups the singular set of the scenario of $\mathfrak{Q}_1$ has become empty. Denote by $C$ again the scenario of $\mathfrak{Q}$ obtained by the blowups, and by $m$ the lifted \factor{} of $C$ as prescribed by Rule~\ref{rul:comm}, Issue~\ref{srul:qtr}.\footnote{Observe that $C$ is not unique but given by Mephisto's responses to Dido's blowup moves.}

As the scenario of $\mathfrak{Q}_1$ is empty, Rule~\ref{rul:quot}, Issue~\ref{srul:ord}, 
implies that $\ord(s)<m(s)+q$ for all singular nodes $s$ of $C$. Hence the maximal value of $\ord(s)-m(s)$ has dropped below $q$. As it is a multiple of $\frac{1}{B}$, it may only drop a finite number of times; eventually it reaches zero, which means that $m$ has become a complete \factor{} of the scenario $C$ of $\mathfrak{Q}$. Then $\mathfrak{Q}$ can be won by Lemma~\ref{lem:mon}.
\end{proof}

%-------------------------------------------------

\begin{rem}
If $n=0$, then there are no tight scenarios of dimension $n$, by Rule~\ref{rul:scen},
Issue~\ref{srul:out}. The proof of Lemma~\ref{lem:rtight} can easily be adapted to show that there is always a strict winning strategy for quests of dimension~0.
\end{rem}

%-------------------------------------------------

\begin{lem}[Descent] \label{lem:descent} 
Let $n>0$. If there is a strict winning strategy for quests of dimension $n-1$,
then there is a strict winning strategy for tight quests of dimension $n$.
\end{lem}

%-------------------------------------------------

\begin{proof}
Let $\mathfrak{Q}$ be a tight quest of dimension $n$, with scenario $C$.
Assume that $C$ has singular nodes, i.e., that the quest is not yet won.
Recall that a critical set for $C$ is a subset ${\cal K}$ of ${\cal H}$ for which
there is a singular node of $C$ below all nodes in
${\cal K}$. Let now $G$ be the collection of all critical sets of $C$, and 
let ${\cal L}={\cal H}$ be the set of all jibs of $C$.\footnote{We need to introduce a new letter here because ${\cal H}$ will change under blowup, whereas ${\cal L}$ will not change.} For any set ${\cal K}\in G$, we define successively three quests: the transversality quest $\mathfrak{P}_{\cal K}$ of $\mathfrak{Q}$, the relaxation quest $\mathfrak{R}_{\cal K}$  of $\mathfrak{P}_{\cal K}$ obtained by releasing ${\cal L}$ from $\mathfrak{P}_{\cal K}$, and the descent quest $\mathfrak{Q}_{\cal K}$ of $\mathfrak{R}_{\cal K}$.

Note that, by assumption, there exists a winning strategy 
for $\mathfrak{Q}_\emptyset$. But it would be too hasty to try to win 
$\mathfrak{Q}_\emptyset$ now, because many nodes that are admissible 
for $\mathfrak{Q}_\emptyset$ and might be needed as 
centers to win it may not be admissible for $\mathfrak{Q}$, so that the winning strategy 
for $\mathfrak{Q}_{\emptyset}$ would lead to the loss of $\mathfrak{Q}$.

In view of this, we follow a procedure which strictly wins the quests $\mathfrak{Q}_{\cal K}$ for ${\cal K}\in G$ one by one, starting with a maximal ${\cal K}$. In the course of this procedure, it may be necessary to blow up nodes that are not singular but only admissible for the involved quests. In other words it is not possible to achieve simultanuous strict winnings for all quests.\footnote{Compare this argument with the treatment of the transversality problem in \cite{Encinas_Hauser:02} where the transversality ideal is defined as a product of ideals. Here, however, we take a construction which mimics the sum of ideals.}  

As transversality and descent quests do not affect the transversal nodes,
$\mathfrak{Q}$ and $\mathfrak{P}_{\cal K}$ have the same transversal set ${\cal T}$, as well as $\mathfrak{R}_{\cal K}$ and $\mathfrak{Q}_{\cal K}$, whose transversal set will be denoted by ${\cal T}_{\cal K}$. Similarly, $\mathfrak{P}_{\cal K}$, $\mathfrak{R}_{\cal K}$ and $\mathfrak{Q}_{\cal K}$ all have the same singular set ${\cal S}_{\cal K}$. 

Assume that ${\cal K}$ is maximal and let $z\in{\cal T}_{\cal K}\cap{\cal S}_{\cal K}$. We claim that $z$ is an admissible center for $\mathfrak{Q}$ and for the quests $\mathfrak{P}_{{\cal K}'}$, $\mathfrak{R}_{{\cal K}'}$ and $\mathfrak{Q}_{{\cal K}'}$, for all ${\cal K}'\in G$. 

Since $z\in{\cal S}_{\cal K}$, it follows that $z\le h$ for all $h\in {\cal K}$. By maximality of ${\cal K}$, $z$ is remote from ${\cal L}\setminus {\cal K}$. Hence $z\in{\cal T}$, by Rule~\ref{rul:relax}, Issue~\ref{srul:relax,T}. From ${\cal S}_{\cal K}\subseteq{\cal S}$ it follows that $z$ is admissible for $\mathfrak{Q}$. 

Let now ${\cal K}'\in G$ be some other critical set. Then ${\cal T}\subseteq{\cal T}_{{\cal K}'}$, hence $z\in{\cal T}_{{\cal K}'}$. We distinguish two cases. If ${\cal K}'\subseteq {\cal K}$, then ${\cal S}_{\cal K}\subseteq{\cal S}_{{\cal K}'}$, hence $z\in{\cal S}_{{\cal K}'}$ and $z$ is admissible for
$\mathfrak{P}_{{\cal K}'}$, $\mathfrak{R}_{{\cal K}'}$, and $\mathfrak{Q}_{{\cal K}'}$.
Otherwise, choose some $h'\in {\cal K}'\setminus {\cal K}$.
By maximality of ${\cal K}$, the node $z$ is remote from $h'$, and therefore
it is also remote from ${\cal S}_{\cal K}$. Again, it follows that $z$ is admissible for
$\mathfrak{P}_{{\cal K}'}$, $\mathfrak{R}_{{\cal K}'}$, and $\mathfrak{Q}_{{\cal K}'}$.

By assumption, there is a strict winning strategy for $\mathfrak{Q}_{\cal K}$.
Dido applies it to strictly win $\mathfrak{Q}_{\cal K}$. By the above observations (still assuming that ${\cal K}$ is maximal), the chosen centers are also admissible for the quests  $\mathfrak{Q}$, $\mathfrak{P}_{{\cal K}'}$, $\mathfrak{R}_{{\cal K}'}$, and $\mathfrak{Q}_{{\cal K}'}$ for all ${\cal K}'\in G$. Therefore they remain open while winning $\mathfrak{Q}_{\cal K}$. And when $\mathfrak{Q}_{\cal K}$ is won,
there will be no node below all jibs in ${\cal K}$ which is singular for $\mathfrak{Q}$.
At this point, we remove ${\cal K}$ from the collection $G$ of critical sets,
take another maximal critical set in $G$ in place of ${\cal K}$ and repeat.
The reasoning that none of the remaining quests becomes invalid still applies.

The last maximal set which Dido removes from $G$ using the above strategy will be the emptyset $\emptyset$.
By then, there is no singular node at all, 
which means that $\mathfrak{Q}$ is resolved.
\end{proof}

%-------------------------------------------------

\begin{thm}
There is a strict winning strategy for Salmagundy.
\end{thm}

\begin{proof}
This is now an obvious consequence of the two lemmata and the remark above.
\end{proof}

\goodbreak

%-------------------------------------------------
%            BIBLIOGRAPHY
%-------------------------------------------------

%-------------------------------------------------
%        ADDRESSES
%-------------------------------------------------
\

\noindent 
Herwig Hauser\\
Fakult\"at f\"ur Mathematik\\
Universit\"at Wien, Austria\\
herwig.hauser@univie.ac.at\\

\noindent
Josef Schicho \\
Johann Radon Institut\\
\"Osterreichische Akademie der Wissenschaften, Linz, Austria\\
josef.schicho@oeaw.ac.at


\begin{thebibliography}{10}

\bibitem{Abhyankar:66}
S.~Abhyankar. 
\newblock {\em Resolution of singularities of embedded algebraic surfaces}.
\newblock Acad. Press 1966. 2nd edition, Springer 1998.

\bibitem{Aroca_Hironaka_Vicente:75}
J.~M. Aroca, H.~Hironaka, and J.~L. Vicente.
\newblock {\em The theory of the maximal contact.}  
\newblock Memorias de Matem{\'a}tica del Instituto ``Jorge Juan'' de Matem\'aticas, Consejo Superior de Investigaciones Cient\'\i ficas, Madrid, No. 29, 1975. 

\bibitem{Aroca_Hironaka_Vicente:77}
J.~M. Aroca, H.~Hironaka, and J.~L. Vicente.
\newblock {\em Desingularization theorems.}
\newblock Memorias de Matem{\'a}tica del Instituto ``Jorge Juan'' de Matem\'aticas, Consejo Superior de Investigaciones Cient\'\i ficas, Madrid, No. 30, 1977. 

\bibitem{Benito_Villamayor:10}
A.~Benito and O.~Villamayor.
\newblock Singularities in positive characteristic: elimination and monoidal transformations. \newblock arXiv:math/0811.4148.

\bibitem{Bennett:70}
B.~Bennett.
\newblock On the characteristic function of a local ring. 
\newblock{\em Ann. Math.}, 91:25--87, 1970.

\bibitem{Bierstone_Milman:90}
E.~Bierstone and P.~Milman.
\newblock Local resolution of singularities.
\newblock In {\em Real analytic and algebraic geometry ({T}rento, 1988)}, {\em Lecture Notes in Math.}, vol. 1420, pages 42--64. Springer, 1990.

\bibitem{Bierstone_Milman:91}
E.~Bierstone and P.~Milman.
\newblock A simple constructive proof of canonical resolution of singularities.
\newblock In T.~Mora and C.~Traverso, editors, {\em Effective methods in
  algebraic geometry}, pages 11--30. {Birkh\"{a}user}, 1991.

\bibitem{Bierstone_Milman:97}
E.~Bierstone and P.~Milman.
\newblock Canonical desingularization in characteristic zero by blowing up the
  maximum strata of a local invariant.
\newblock {\em Invent. math.}, 128:207--302, 1997.

\bibitem{Bierstone_Milman_Temkin:09} 
E.~Bierstone and P.~Milman, M.~Temkin.
\newblock $Q$-universal desingularization. 
\newblock arXiv:0905.3580.

\bibitem{Bloch:94}
S.~Bloch.
\newblock The moving lemma for higher Chow groups.
\newblock {\em J. Alg. Geom.}, 3:537--568,  1994.

\bibitem{Bodnar:03}
G.~Bodn\'{a}r.
\newblock Computation of blowing up centers.  
\newblock {\em J. Pure Appl. Algebra}, 179:221--233, 2003.

\bibitem{Bodnar_Schicho:99f}
G.~Bodn\'{a}r and J.~Schicho.
\newblock Automated resolution of singularities for hypersurfaces.
\newblock {\em J. Symb. Comp.}, 30:401--428, 2000.

\bibitem{Bravo_Villamayor:10}
A.~Bravo and O.~Villamayor.
\newblock Singularities in positive characteristic,  stratification and simplification of the singular locus.
\newblock{\em Adv.÷ Math.}, 224:1349--1418, 2010.

\bibitem{Cutkosky:04}
D.~Cutkosky.
\newblock {\em Resolution of Singularities.} 
\newblock Graduate Studies in Math., vol.  63. Amer.÷ Math.÷ Soc.÷ 2004.

\bibitem{Encinas_Hauser:02}
S.~Encinas and H.~Hauser.
\newblock Strong resolution of singularities in characteristic zero.
\newblock {\em Comment. Math. Helv.}, 77:421--445, 2002.

\bibitem{Encinas_Villamayor:97b}
S.~Encinas and O.~Villamayor.
\newblock A course on constructive desingularization and equivariance.
\newblock In H.~Hauser, editor, {\em Resolution of Singularities ({O}bergurgl,
  1997)}, pages 147--227. {Birkh\"{a}user}, 2000.

\bibitem{Encinas_Villamayor:06}
S.~Encinas and O.~Villamayor.
\newblock Rees algebras and resolution of singularities.
\newblock {\em Rev.÷ Mat.÷ Iberoam.} Proc. XVI-Coloquio Latinoamericano de \'Algebra 2006. 

\bibitem{Fruehbis-Krueger_Pfister:03}
A.~Fr\"uhbis-Kr\"uger and G.~Pfister.
\newblock Computational aspects of singularities.
\newblock In {\em Singularities in geometry and topology}, pages 253--327.
  World Sci. Publ., Hackensack,  2007.

\bibitem{Goethe}
J.~W. von Goethe.
\newblock {\em Faust. Der {T}rag\"odie erster {T}eil}.
\newblock 1808.

\bibitem{Hauser:10}
H.~Hauser.
\newblock On the problem of resolution of singularities in positive
  characteristic ({O}r: a proof we are still waiting for).
\newblock {\em Bull. Amer. Math. Soc.}, 47(1):1--30, 2010.

\bibitem{Hironaka:64}
H.~Hironaka.
\newblock Resolution of singularities of an algebraic variety over a field of
  characteristic 0.
\newblock {\em Ann. Math.}, 79:109--326, 1964.

\bibitem{Hironaka:67}
H.~Hironaka.
\newblock Characteristic polyhedra of singularities.
\newblock {\em J. Math. Kyoto Univ.}, 7:251--293, 1967.

\bibitem{Hironaka:77}
H.~Hironaka.
\newblock Idealistic exponents of singularity.
\newblock {\em The Johns Hopkins Centennial Lectures},
Johns Hopkins University Press 1977.

\bibitem{Hironaka:07}
H.~Hironaka.
\newblock Theory of infinitely near singular points. 
\newblock {\em J. Korean Math. Soc.}, 40:901--920, 2003.

\bibitem{Kawanoue:07}
H.~Kawanoue.
\newblock Toward resolution of singularities over a field of positive characteristic. Part I. \newblock {\em  Publ.~Res.~Inst.~Math.~Sci.}, 43:819--909, 2007.

\bibitem{Kawanoue_Matsuki:08}
H.~Kawanoue  and K.~Matsuki. 
\newblock Toward resolution of singularities over a field of positive characteristic (The Idealistic Filtration Program). Part II. Basic invariants associated to the idealistic filtration and their properties.
\newblock {\em Publ.~Res.~Inst.~Math.~Sci.}, 46:819--909, 2010.

\bibitem{Kollar:07}
J.~Koll{\'a}r.
\newblock {\em Lectures on resolution of singularities}.  
{\em Annals of Mathematics Studies}, vol. 166.
\newblock Princeton University Press  2007.

\bibitem{Levine:01}
M.~Levine. 
\newblock Blowing up monomial ideals.  
\newblock {\em J. Pure and Applied Algebra}, 160:67--103, 2001.

\bibitem{Lipman:75}
J.~Lipman.
\newblock Introduction to resolution of singularities. 
\newblock {\em Proc. Symp. Pure Appl. Math. Amer. Math. Soc.÷}, 29:187--230, 1975.

\bibitem{Spivakovsky:82}
M.~Spivakovsky.
\newblock A counterexample to {Hironaka}'s ``hard'' polyhedral game.
\newblock {\em Publ. Res. Inst. Math. Sci.}, 18:1009--1012, 1982.

\bibitem{Spivakovsky:83}
M.~Spivakovsky.
\newblock A solution to {Hironaka}'s polyhedral game.
\newblock In M.~Artin and J.~Tate, editors, {\em Arithmetic and geometry},
  pages 419--432. Birkh{\"a}user, 1983.

\bibitem{Villamayor:89}
O.~Villamayor.
\newblock Constructiveness of {Hironaka's} resolution.
\newblock {\em Ann. Scient. Ecole Norm. Sup. 4} 22:1--32, 1989.

%\bibitem{Villamayor:91}
%O.~Villamayor.
%\newblock Introduction to the algorithm of resolution.
%\newblock In {\em Algebraic geometry and singularities ({La R\'abida} 1991)}, pages 123--154. {Birkh\"{a}user}, 1996.

\bibitem{Villamayor:92}
O.~Villamayor.
\newblock Patching local uniformizations.
\newblock {\em Ann. Scient. Ec. Norm. Sup.}, 25:629--677, 1992.

\bibitem{Villamayor:07}
O.~Villamayor.
\newblock Hypersurface singularities in positive characteristic.
\newblock {\em Adv. Math.}, 213:687--733, 2007.

\bibitem{Wlodarczyk:05}
J.~W{\l}odarczyk.
\newblock Simple {H}ironaka resolution in characteristic zero.
\newblock {\em J. Amer. Math. Soc.}, 18(4):779--822, 2005.

\bibitem{Zeillinger:06}
D.~Zeillinger.
\newblock  Polyhedral games and resolution of singularities.
\newblock PhD thesis, Univ. Innsbruck, 2006.

\end{thebibliography}
\end{document}